\documentclass[aos]{imsart}
\usepackage{microtype}
\usepackage{blkarray}
\usepackage{xspace}
\usepackage{xfrac}
\usepackage{amsthm,amsmath,amsfonts,amssymb}
\usepackage[numbers]{natbib}
\usepackage{xcolor}
\usepackage[colorlinks,citecolor=blue,urlcolor=blue]{hyperref}%
\usepackage[capitalise]{cleveref} %
\usepackage{enumitem}

\newlist{inparaenum}{enumerate*}{1}
\setlist[inparaenum]{label=(\arabic*)}

\makeatletter
\def\journal@name{}%
\def\journal@url{}%
\makeatother

\usepackage{soul}

\usepackage{caption}
\usepackage{subcaption}
\usepackage{tikz}
\usetikzlibrary{arrows.meta}
\usetikzlibrary{positioning}

\startlocaldefs
\theoremstyle{plain}
\newtheorem{theorem}{Theorem}[section]

\newtheorem{corollary}[theorem]{Corollary}
\newtheorem{lemma}[theorem]{Lemma}
\newtheorem{proposition}[theorem]{Proposition}
\newtheorem{problem}[theorem]{Problem}

\theoremstyle{definition}
\newtheorem{remark}[theorem]{Remark}
\newtheorem{example}[theorem]{Example}

\newtheorem{definition}[theorem]{Definition}

\newtheorem*{convention*}{Convention}

\newcommand{\rr}{\mathbb{R}}

\newcommand{\cm}{\mathcal{M}}
\newcommand{\cn}{\mathcal{N}}
\newcommand{\ct}{\mathcal{T}}
\newcommand{\cg}{\mathcal{G}}

\newcommand{\PD}{\mathrm{PD}}

\DeclareMathOperator{\pa}{pa}

\DeclareMathOperator{\neighbors}{ne}
\DeclareMathOperator{\ch}{ch}

\let\vec\undefined
\DeclareMathOperator{\vec}{vec}
\DeclareMathOperator{\vech}{vech}
\newcommand{\indep}{\mathrel{\text{$\perp\mkern-10mu\perp$}}}

\newcommand{\nrnodes}{n}

\newcommand{\pd}{\mathrm{PD}}
\newcommand{\ot}{\leftarrow}
\newcommand{\Id}{I}

\usepackage{mathtools}

\newcommand{\defas}{\coloneqq}

\definecolor{benpurple}{RGB}{180, 0, 240}

\endlocaldefs

\begin{document}

\begin{frontmatter}
\title{Structural Identifiability of\\%
Graphical Continuous Lyapunov Models}
\runtitle{Structural Identifiability of GCLMs}

\begin{aug}
\author[A]{\fnms{Carlos}~\snm{Am\'endola}\ead[label=e1]{amendola@math.tu-berlin.de}\orcid{0000-0003-1945-8874}},
\author[B]{\fnms{Tobias}~\snm{Boege}\ead[label=e2]{post@taboege.de}\orcid{0000-0001-7284-1827}}, \\
\author[C]{\fnms{Benjamin}~\snm{Hollering}\ead[label=e3]{benjamin.hollering@mis.mpg.de}\orcid{0000-0003-3803-2879}}, and
\author[D]{\fnms{Pratik}~\snm{Misra}\ead[label=e4]{pmisra@binghamton.edu}\orcid{0009-0006-0011-6232}}
\address[A]{Technical University of Berlin \printead[presep={,\ }]{e1}}
\address[B]{UiT The Arctic University of Norway\printead[presep={,\ }]{e2}}
\address[C]{Max Planck Institute for Mathematics in the Sciences\printead[presep={,\ }]{e3}}
\address[D]{Binghamton University, State University of New York \printead[presep={,\ }]{e4}}
\end{aug}

\begin{abstract}
We prove two characterizations of model equivalence of acyclic graphical
continuous Lyapunov models (GCLMs) with uncorrelated noise. The first
result shows that two graphs are model equivalent if and only if they
have the same skeleton and equivalent induced 4-node subgraphs.
We also give a transformational characterization via structured edge
reversals. The two theorems are Lyapunov analogues of celebrated results for Bayesian networks by Verma and Pearl, and Chickering, respectively.
Our results have broad consequences for the theory of causal inference
of GCLMs. First, we find that model equivalence classes of acyclic GCLMs
refine the corresponding classes of Bayesian networks. Furthermore, we
obtain polynomial-time algorithms to test model equivalence and structural
identifiability of given directed acyclic~graphs.
\end{abstract}



\end{frontmatter}

\section{Introduction}

Stationary diffusion processes have recently emerged as a new framework for modeling causal systems in applications where it is difficult, or even impossible, to observe a system more than once. A classic example of this is cellular biology where measurement often destroys the sample. Thus the data can be thought of as independent observations of the stationary distribution of the diffusion process. More concretely, one assumes that the data is generated by a stationary stochastic process $\{\mathbb{X}(t)\}_{t \geq 0}$ which arises as the solution of the stochastic differential equation (SDE):
\begin{equation}
\label{eqn:sde}
\mathrm{d}\mathbb{X}(t) = f(\mathbb{X}(t)) \mathrm{d}t + D \mathrm{d}\mathbb{W}(t),
\end{equation}
where $f$ is a Lipschitz-continuous map that specifies the drift, $D$ is an invertible matrix, and $\mathbb{W}$ is a standard $n$-dimensional Brownian motion. One advantage of this framework over classical structural causal models is that it provides a natural temporal interpretation for causal feedback~loops.

This novel approach to causality was first developed in \cite{katie2019} and \cite{varando2020graphical} with the restriction to \emph{Ornstein--Uhlenbeck processes}. In this case, the SDE above takes the form
\begin{equation}
\label{eqn:ornstein-uhlenbeck}
\mathrm{d}\mathbb{X}(t) = M \mathbb{X}(t) \mathrm{d}t + D \mathrm{d}\mathbb{W}(t),
\end{equation}
where $M$ is an invertible drift matrix
whose entries control the magnitude of cause-effect relations among the components of~$X$ over time. If $M$ is \emph{stable}, i.e., all of its eigenvalues have strictly negative real part, then this process has a unique stationary distribution. It is Gaussian, centered and with covariance matrix $\Sigma$ which is the unique positive definite solution to the \emph{continuous Lyapunov equation}
\begin{equation}
  \label{eq:Lyapunov}
  M\Sigma + \Sigma M^T + DD^T = 0.
\end{equation}
This shows that the stationary distribution depends on $D$ only through the volatility matrix $C = DD^T$.
In the spirit of graphical modeling, we restrict the cause-effect relations in the SDE by imposing a selection of constraints of the form $m_{ji} = 0$. These zeros are encoded in the non-edges of a directed graph~$\cg$. The resulting set of all covariance matrices $\Sigma$ which arise as solutions to the Lyapunov equation for $\cg$-sparse stable matrices $M$ is called the \emph{graphical continuous Lyapunov model} (GCLM) associated to~$\cg$.

Graphical continuous Lyapunov models have received growing interest in recent years as a natural analogue of structural equation models (SEM) to the stochastic process setting. Indeed the SDE in \cref{eqn:sde} is a direct analogue of the structural equation system $X = F(X) + \varepsilon$ for a random vector $X$ and function $F$. If $F(X) = \Lambda X$ is linear and $\Lambda$ is $\cg$-sparse then each $X_i$ is specified as a linear function of its parents plus independent noise. This is the classical setting of Gaussian Bayesian networks. %
The Ornstein--Uhlenbeck process \eqref{eqn:ornstein-uhlenbeck} is entirely analogous: the drift $f(X) = M X$ is linear and the noise from $\mathrm{d} \mathbb W(t)$ is uncorrelated provided that $D$ is diagonal. %

\subsection*{Summary of the Main Contributions}

In this paper, we study the structural identifiability and model equivalence of GCLMs. Two graphs $\cg_1$ and $\cg_2$ are said to be \emph{model equivalent} in the Lyapunov sense if their corresponding GCLMs are equal, that is $\cm_{\cg_1} = \cm_{\cg_2}$. We~provide a complete characterization of the \emph{model equivalence classes} of GCLMs in the case when the matrix $D$ is diagonal, i.e., the noise terms from the Brownian motion are uncorrelated. This includes an explicit characterization of when two graphs define equivalent models, which is the analogue of the well-known result of Verma and Pearl \cite{andersson1997characterization, VermaPearl} that two DAGs define equivalent Bayesian network models if and only if they have the same skeleton and v-structures. We also give a \emph{transformational characterization} of model equivalence which is analogous to that for Bayesian networks due to Chickering~\cite{chickering1995transformational}.

An immediate consequence of our results is that the equivalence classes of Lyapunov models are a \emph{refinement} of those of d-separation, despite the lack of conditional independence that Lyapunov models exhibit.
We note that since Lyapunov models satisfy only marginal independence statements, conditional independence is almost useless as a tool for characterizing when two models are equivalent. This is because the only structure which can be identified using only marginal independencies is the so-called \emph{trek graph} of $\cg$ which is a bidirected graph with an edge between $i$ and $j$ if and only if there exists a trek from $i$ to $j$ in $\cg$. Thus only using the conditional independence information to identify structure would yield a set of model equivalence classes which are much coarser than d-separation. Nevertheless, our results show that Lyapunov models actually encode more fine-grained causal information than Bayesian networks. To obtain these stronger results, we develop several new tools for distinguishing statistical models. Some of these results are specific to the Lyapunov setting such as the \emph{missing-edge relations} which appear in \cref{def: missing edge}. These relations are rank constraints on an associated matrix $A_\cg(\Sigma)$ which encodes the \emph{vectorized} form of Lyapunov equation which was introduced in \cite{dettling2022identifiability}. We show that these relations provide an \emph{implicit description} of the model similar to the well-known description of Bayesian networks via conditional independence~\cite{lauritzen1996graphical}. In addition to these relations, we also show how proving parameter identifiability results can be used to infer structural identifiability results.

\subsection*{Outline}

The remainder of the paper is organized as follows. In \Cref{sec:model} we introduce the Lyapunov graphical models along with basic properties and relevant definitions. In \Cref{sec:main} we present our main results on structural identifiability of the underlying directed graph, including a transformational characterization of model equivalence. \Cref{sec:proofs} is devoted to detailed proofs and explanations of the mathematical tools we use. In \Cref{sec:outlook} we end with a discussion on how some of the tools we develop here could be used to study more complicated models,  such as GCLMs on mixed graphs, and discuss some of the difficulties in this setting.

\section{Graphical Continuous Lyapunov Models}\label{sec:model}

\subsection{Definition of the Model}

A graphical continuous Lyapunov model imposes qualitative restrictions on the drift matrix~$M$ in \cref{eqn:ornstein-uhlenbeck}. These restrictions are encoded in a directed graph $\cg = (V, E)$ with the familiar interpretation from Bayesian networks: if $i \to j \notin E$ then $m_{ji} = 0$ and thus $\mathbb{X}_i(t)$ does not directly influence $\mathbb{X}_j(t)$. We will usually use $V = \{1,\dots,\nrnodes\}$ and denote by $\rr^E$ the set of all real $\nrnodes \times \nrnodes$-matrices $M$ with $m_{ji} = 0$ if $i \to j \notin E$.

\begin{definition}
    Let $\cg = (V, E)$ be a directed graph with vertex set $V$ and edge set $E$ which includes all self-loops $i \to i$, $i \in V$. Let $\PD_\nrnodes$ denote the set of $n\times n$ positive definite matrices. For any fixed $C \in \PD_\nrnodes$, the \emph{graphical continuous Lyapunov model} (GCLM) of $\cg$ is the family of all multivariate normal distributions $\cn(0,\Sigma)$ with covariance matrix $\Sigma$ in the~set
    \[
    \cm_{\cg,C} = \{ \Sigma \in \mathrm{PD}_\nrnodes : M \Sigma + \Sigma M^T + C = 0 \text{ for a stable matrix } M \in \rr^E\}.
    \]
\end{definition}

We now fix terminology and notation pertaining to matrices and graphs.
The $\nrnodes$-dimensional identity matrix is $\Id_\nrnodes$.
Given an $n \times n$-matrix $\Sigma$ and subsets $A, B \subseteq \{1, \dots, n\}$ the submatrix of $\Sigma$ with rows $A$ and columns $B$ is denoted by $\Sigma_{A,B}$; if the two index sets are equal, we write more shortly $\Sigma_A = \Sigma_{A,A}$.

A graph $\cg = (V, E)$ is always understood as a directed graph. Its edge set $E$ is also denoted by~$E(\cg)$. The graph $\cg$ is \emph{simple} if for all distinct $i, j \in V$ it contains at most one of the edges $i \to j$ and $j \to i$. For every vertex $i \in V$ there is a self-loop $i \to i \in E$. For simplicity, we will suppress the self-loops when drawing graphs. Furthermore, self-loops are tacitly ignored when talking about the \emph{parents} $\pa_\cg(i)$, the \emph{children} $\ch_\cg(i)$, the \emph{neighbors} $\neighbors_\cg(i) = \pa_\cg(i) \cup \ch_\cg(i)$, and the notion of acyclicity. Thus, a graph is acyclic (and therefore a DAG) if it does not contain a simple directed cycle of length greater than one. Self-loops are, however, elements of the edge set and so a complete DAG has $\binom{n+1}{2} = n + \binom{n}{2}$ edges.
A \emph{trek} between two nodes $i, j \in V$ is a sequence of edges $i = i_l \ot i_{l-1} \ot \dots \ot i_1 \ot i_0 = j_0 \to j_1 \to \dots \to j_r = j$. The vertex $i_0=j_0$ at which the direction of the arrows changes is called the \emph{top} of the trek. We allow $l=0$ and $r=0$ so that a directed path from $i$ to $j$ or from $j$ to $i$ is a trek.

\subsection{Vectorizations of the Lyapunov Equation}

With $C$ fixed, the Lyapunov equation \eqref{eq:Lyapunov} is linear in either $\Sigma$ or $M$. Therefore it can be turned into a linear system of equations for either in terms of the other through vectorization. This amounts, in one direction, to effectively solving the SDE~\eqref{eqn:ornstein-uhlenbeck} by obtaining the stationary covariance matrix $\Sigma$ directly from the drift~$M$; in the other direction it shows that the drift coefficients are globally identifiable from (the second moments of) the distribution. This observation was first utilized in~\cite{dettling2022identifiability}.

Let $\vec(C) = (c_{kl} : 1 \le k, l \le \nrnodes) \in \rr^{\nrnodes^2}$ denote the (column-wise) vectorization of a matrix. If we regard $M$ as fixed, a matrix $\Sigma$ solves \eqref{eq:Lyapunov} if and only if it solves the vectorized equation
\begin{equation}
  \label{eqn:BM}
  B(M) \vec(\Sigma) = -\vec(C),
\end{equation}
where $B(M) = (\Id_\nrnodes \otimes M + M \otimes \Id_\nrnodes)$ and $\otimes$ denotes the Kronecker product. Since $M$ is stable, $B(M)$ is invertible (in fact, also stable). Thus, Cramer's rule solves for the entries of~$\Sigma$ in terms of~$M$ and~$C$. Specifically, every entry $\sigma_{ij}$ can be written as $\frac{\det(B_{ij}(M; C))}{\det(B(M))}$, where $B_{ij}(M; C)$ is the matrix obtained from $B(M)$ by replacing the $(i,j)$-column by $-\vec(C)$.

On the other hand, if $\Sigma$ is fixed then a matrix $M$ solves \cref{eq:Lyapunov} if and only if it solves the vectorized equation
\begin{equation}
  \label{eq:LyapVecM}
  (\Sigma \otimes \Id_\nrnodes + (\Id_\nrnodes \otimes \Sigma)K_\nrnodes) \vec(M) = -\vec(C),
\end{equation}
where $K_\nrnodes$ is the commutation matrix sending $\vec(A)$ to $\vec(A^T)$ for any $\nrnodes\times\nrnodes$-matrix~$A$. However, since $\Sigma$ and $C$ are symmetric matrices, the matrix $(\Sigma \otimes \Id_\nrnodes + (\Id_\nrnodes \otimes \Sigma)K_\nrnodes)$ has many redundant rows. We use the following subsystem from~\cite{dettling2022identifiability}:

\begin{definition} \label{def:Asigma}
Let $\cg = (V, E)$ be a graph and consider the $\binom{\nrnodes+1}{2} \times |E|$ matrix $A_\cg(\Sigma)$ whose rows are indexed by pairs $(k,l)$ with $k \le l$ and columns by edges $i \to j \in E$. Its entries are given by
\begin{equation}
\label{eqn:AsigmaEntries}
   A_\cg(\Sigma)_{(k,l),i\to j}=
   \begin{cases}
   0, \quad &\text{if} \quad j \neq k,l; \\
   \sigma_{li}, \quad &\text{if} \quad j= k,\, k \neq l;\\
   \sigma_{ki}, \quad &\text{if} \quad j=l, \, l \neq k;\\
   2\sigma_{ji}, \quad &\text{if} \quad j=k=l.
   \end{cases}
\end{equation}
We write $A(\Sigma)$ when $E = V \times V$ includes every possible edge.
\end{definition}

Then $M$ solves the vectorized Lyapunov equation~\eqref{eq:LyapVecM} if and only if $M$ solves %
\begin{equation}
\label{eq:Asigma}
A(\Sigma) \vec(M) = -\vech(C),
\end{equation}
where $\vech(C) = (c_{kl} : k \leq l)$ is the \emph{half-vectorization} of the symmetric matrix $C$.

\begin{example} \label{eg:Asigma}
For $n=3$ nodes the full $A(\Sigma)$ matrix is
\begin{align*}
\begin{blockarray}{@{\hspace{10pt}}c@{\hspace{10pt}}c@{\hspace{10pt}}c@{\hspace{10pt}}c@{\hspace{10pt}}c@{\hspace{10pt}}c@{\hspace{10pt}}c@{\hspace{10pt}}c@{\hspace{10pt}}c@{\hspace{10pt}}c}
& \text{\small$1\to 1$} & \text{\small$1\to 2$} & \text{\small$1\to 3$} & \text{\small$2 \to 1$} & \text{\small$2 \to 2$} & \text{\small$2 \to 3$} & \text{\small$3 \to 1$} & \text{\small$3 \to 2$} & \text{\small$3 \to 3$} \\
\begin{block}{c(ccccccccc)}
    \text{\small$(1,1)$}&2 \sigma_{11} &0&0&2 \sigma_{12} & 0 &0 & 2\sigma_{13} & 0 & 0 \\
    \text{\small$(1,2)$}&\sigma_{12} & \sigma_{11} &0& \sigma_{22}  & \sigma_{12} &0& \sigma_{23} & \sigma_{13} &0\\
    \text{\small$(1,3)$}&\sigma_{13} &0&\sigma_{11}& \sigma_{23} & 0 & \sigma_{12}& \sigma_{33} & 0 & \sigma_{13}\\
    \text{\small$(2,2)$}&0 &2\sigma_{12}& 0& 0 & 2 \sigma_{22} &0& 0 & 2\sigma_{23} & 0 \\
    \text{\small$(2,3)$}&0 & \sigma_{13}&\sigma_{12}& 0 & \sigma_{23}& \sigma_{22}   & 0 & \sigma_{33} & \sigma_{23}\\
    \text{\small$(3,3)$}&0 & 0&2\sigma_{13}& 0 & 0&2\sigma_{23} & 0 & 0 & 2 \sigma_{33}\\
\end{block}
\end{blockarray} \; .
\end{align*}
For every graph $\cg = (V, E)$ the matrix $A_\cg(\Sigma)$ is the submatrix of $A(\Sigma)$ which selects the columns indexed by the edges~$E$ present in~$\cg$.
\end{example}

Note that $A_\cg(\Sigma)$ is only a square matrix if the graph $\cg$ has $\binom{n+1}{2}$ edges. Thus if $\cg$ is a simple graph (and includes all self-loops on nodes), this holds if and only if $\cg$ is complete. Identifiability of $M$ from $\Sigma$ hinges on the rank of this matrix.

\begin{theorem}[{\cite[Lemma 4.3, Theorem 7.1]{dettling2022identifiability}}]
\label{thm:LyapIdentifiability}
Let $\cg = (V, E)$ be a simple graph. Then $A_\cg(\Sigma)$ is full rank, implying that the model $\cm_{\cg, C}$ is \emph{globally identifiable} for all $C \in \pd_\nrnodes$. In~particular, if $\cg$ is simple and complete then
\begin{equation}
  \label{eq:LyapIdentifiability}
  m_{ji} = \frac{\det(A_\cg^{ij}(\Sigma; C))}{\det(A_\cg(\Sigma))}
\end{equation}
holds for any $M \in \rr^E$ and $\Sigma \in \cm_{\cg,C}$ solving~\eqref{eq:Lyapunov}. Here, $A_\cg^{ij}(\Sigma; C)$ is the matrix $A_\cg(\Sigma)$ in which the $(i,j)$-column is replaced by $-\vech(C)$.
\end{theorem}

\subsection{Missing-edge Relations}

In a graphical continuous Lyapunov model, the solution $M$ to \eqref{eq:Asigma} is subject to sparsity constraints given by the graph~$\cg$. The existence of a $\cg$-sparse solution to \eqref{eq:Asigma} puts additional (non-linear) constraints on $A_\cg(\Sigma)$ and thus ultimately on $\Sigma$. Our first observation is that these non-linear constraints characterize the GCLM.
\begin{definition} \label{def: missing edge}
Let $\cg$ be a simple directed graph and $C \in \PD_\nrnodes$. The graph $\cg'$ on the same vertex set is a \emph{completion} of $\cg$ if $\cg'$ is simple, $E(\cg) \subseteq E(\cg')$ and $|E(\cg')| = \binom{n+1}{2}$. To~any edge $i \to j \in E(\cg') \setminus E(\cg)$ added in the completion, there is a \emph{missing-edge relation}
\begin{equation}
  \label{eq:MissingEdge}
  \det(A_{\cg'}^{ij}(\Sigma; C))=0.
\end{equation}
The function of $\Sigma$ which appears on the left-hand side of \eqref{eq:MissingEdge} %
is denoted by $f^{ij}_{\cg'}(\Sigma; C)$.
\end{definition}
It is easy to see that every simple graph has a completion.

\begin{proposition} \label{prop:MissingEdge}
Let $\cg = (V, E)$ be a simple directed graph, $C \in \PD_\nrnodes$, and $\cg'$ a completion of~$\cg$. A positive definite matrix $\Sigma$ belongs to $\cm_{\cg,C}$ if and only if $f^{ij}_{\cg'}(\Sigma; C) = 0$ for every missing-edge relation with respect to~$\cg'$.
\end{proposition}

\begin{proof}
Recall that the Lyapunov equation \eqref{eq:Lyapunov} establishes a bijection between stable drift matrices $M \in \rr^E$ and positive definite covariance matrices $\Sigma \in \cm_{\cg,C}$. Let $E'$ be the edge set of $\cg'$. Then $\rr^E$ is naturally a subspace of $\rr^{E'}$ defined by $m_{ji} = 0$ for $i \to j \in E' \setminus E$.

If $\Sigma \in \cm_{\cg,C}$ then $M$ is the unique solution to $A_{\cg}(\Sigma) \vec(M) = -\vech(C)$. Equivalently, it is the unique solution to $A_{\cg'}(\Sigma) \vec(M) = -\vech(C)$ with $m_{ji} = 0$ for $i \to j \in E' \setminus E$. By \Cref{thm:LyapIdentifiability}, $A_{\cg'}(\Sigma)$ is an invertible square matrix and Cramer's rule shows $m_{ji} = f^{ij}_{\cg'}(\Sigma; C) / \det(A_{\cg'}(\Sigma))$. Hence, $m_{ji} = 0$ if and only if $f^{ij}_{\cg'}(\Sigma; C) = 0$.

Conversely, if $\Sigma$ is positive definite then there exists a unique stable matrix $M$ such that $(M, \Sigma)$ together satisfy the Lyapunov equation~\eqref{eq:Lyapunov}. This $M$ is the (necessarily unique) solution to the linear system $A_{\cg'}(\Sigma) \vec(M) = -\vech(C)$. Since $\Sigma$ satisfies all missing-edge relations with respect to $\cg'$, Cramer's rule implies $m_{ji} = 0$ for $i \to j \in E' \setminus E$. Thus $\Sigma$ is obtained as the positive definite solution to \eqref{eq:Lyapunov} for some stable $M \in \rr^E$ which proves $\Sigma \in \cm_{\cg,C}$.
\end{proof}

\begin{remark} \label{rem:MissingEdge}
The same general approach to characterizing graphical models is also at work in the realm of Bayesian networks but is often stated in terms of conditional independence. Namely, if $i \to j$ is a missing edge in a directed acyclic graph with $i < j$ then the conditional independence $[i \indep j \mid \pa(j)]$ holds in the Bayesian network. This, in turn, imposes an equation on all covariance matrices associated to the model, namely the vanishing of the partial correlation of $i$ and $j$ given~$\pa(j)$. It is well-known that this \emph{ordered pairwise Markov property} characterizes the Gaussian Bayesian network. In this sense, our relations in \Cref{prop:MissingEdge} are to Lyapunov models what conditional independence is to Bayesian networks. It remains a challenge, however, to ascribe a statistical interpretation to the missing-edge relations.
\end{remark}

To motivate the following developments, consider the \emph{forward path} $1 \to 2 \to 3$ and the \emph{backward path} $1 \ot 2 \ot 3$. A fundamental obstacle in applications of Bayesian networks to causal reasoning is that these two graphs yield the same statistical model (defined by the conditional independence $[1 \indep 3 \mid 2]$) even though they make opposite assertions about causality. In this situation there is no hope to statistically infer the true causal structure from observational data alone. For this example (and indeed many others), Lyapunov models provide a remedy.

\begin{example} \label{eg:123}
The GCLMs corresponding to the forward and the backward $3$-paths are distinguishable using \Cref{prop:MissingEdge}. We fix $C = 2\Id_3$ to illustrate the idea. The missing edge relation of $1 \to 2 \to 3$ with respect to the completion $\cg'$ adding the edge $1 \to 3$ mandates that the matrix $A_{\cg'}^{13}(\Sigma; 2\Id_3)$ is singular for every $\Sigma$ in the forward path GCLM, i.e.,
\begin{equation}
\label{eq:123}
 \det
\begin{blockarray}{@{\hspace{10pt}}c@{\hspace{10pt}}c@{\hspace{10pt}}c@{\hspace{10pt}}c@{\hspace{10pt}}c@{\hspace{10pt}}c@{\hspace{10pt}}c@{\hspace{10pt}}c@{\hspace{10pt}}c@{\hspace{10pt}}c}
& \text{\small$1\to 1$} & \text{\small$1\to 2$} & \text{\small$1\to 3$} & \text{\small$2 \to 2$} & \text{\small$2 \to 3$} & \text{\small$3 \to 3$} \\
\begin{block}{c(ccccccccc)}
    \text{\small$(1,1)$}&2 \sigma_{11} &0&-2 & 0 &0 & 0 \\
    \text{\small$(1,2)$}&\sigma_{12} & \sigma_{11} &0 & \sigma_{12} &0& \sigma_{23} \\
    \text{\small$(1,3)$}&\sigma_{13} &0&0& 0 & \sigma_{12}& \sigma_{33} \\
    \text{\small$(2,2)$}&0 &2\sigma_{12}& -2& 2 \sigma_{22} &0& 0 \\
    \text{\small$(2,3)$}&0 & \sigma_{13}&0 & \sigma_{23}& \sigma_{22}   & 0 \\
    \text{\small$(3,3)$}&0 & 0&-2& 0&2\sigma_{23} & 0 \\
\end{block}
\end{blockarray} \; = 0.
\end{equation}
Almost every covariance matrix in the backward path model will \emph{not} satisfy this equation! For~example, take the stable matrix $M$ below whose sparsity pattern aligns with the backward path and solve the Lyapunov equation for $\Sigma$:
\begingroup\setlength\arraycolsep{4pt}
\begin{equation}
  M = \begin{pmatrix}
    -1 &  1 &  0 \\
     0 & -1 &  1 \\
     0 &  0 & -1
  \end{pmatrix}, \qquad
  \Sigma = \begin{pmatrix}
    \sfrac{15}{8} & \sfrac78 & \sfrac14 \\
    \sfrac78 & \sfrac32 & \sfrac12 \\
    \sfrac14 & \sfrac12 & 1
  \end{pmatrix}.
\end{equation}
\endgroup
The corresponding matrix from \eqref{eq:123} has determinant $\frac{13\,113}{256}$ which shows that $\Sigma$ is in the backward but not in the forward path model, hence distinguishes the two GCLMs.
\end{example}

\section{Main Results}\label{sec:main}

In this section we give formal definitions of \emph{model equivalence} and \emph{structural identifiability} and then state our main identifiability results for GCLMs. Namely, we provide a complete characterization of all structurally identifiable GCLMs given by directed acyclic graphs (DAGs) with uncorrelated noise, and a description of their model equivalence classes via induced subgraphs. This result can be seen as an analogue of the well-known description of the Markov equivalence classes of Bayesian networks given by Verma and Pearl in \cite{VermaPearl}. Additionally, we derive a \emph{transformational characterization} of model equivalence in this setting which describes how a DAG may be transformed into any other DAG in the same model equivalence class by flipping certain edges. This establishes a Lyapunov analogue of Chickering's transformational characterization for Bayesian networks~\cite{chickering1995transformational}.

Our results require uncorrelated noise, i.e., we shall assume that $C$ is diagonal. A simple transformation due to \cite[Proposition~3.7]{dettling2022identifiability} even allows us to fix $C = 2 \Id_\nrnodes$ without loss of generality. We shall abbreviate notation correspondingly, in particular $\cm_{\cg} \defas \cm_{\cg,2 \Id_\nrnodes}$. As~now the GCLMs under consideration are entirely determined by a DAG $\cg$ --- without cycles or hidden confounders --- they invite comparisons to classical Bayesian networks. Some of our results hold for more general $C$ at the expense of becoming more technical to state. We will return to this topic in \Cref{sec:outlook}.

\begin{definition}
Two graphs $\cg_1$ and $\cg_2$ on the same vertex set~$V$ are \emph{model equivalent} (in the Lyapunov sense) if $\cm_{\cg_1} = \cm_{\cg_2}$. This induces an equivalence relation on all graphs over~$V$.
Fix a class $\mathcal{L}$ of graphs over~$V$. A graph $\cg \in \mathcal{L}$ is \emph{structurally identifiable within $\mathcal{L}$} if its model equivalence class intersects $\mathcal{L}$ in exactly one element.
\end{definition}

\begin{remark}
A stronger form of structural identifiability, known as \emph{global structural identifiability}, requires that the models $\{\cm_\cg\}_{\cg \in \mathcal{L}}$ be pairwise disjoint. We cannot hope for global identifiability as (a) every model contains the standard normal distribution, and (b) any orientation $D_\nrnodes$ of the complete graph has as its model $\cm_{D_\nrnodes} = \PD_\nrnodes$ and thus contains every other model. Even besides containment relations, distinct Lyapunov models may also intersect each other in non-trivial ways.
Our definition agrees instead with \emph{generic structural identifiability} (also known as \emph{model distinguishability}) from, e.g.,~\cite{drton2023identifiabilityhomoscedasticlinearstructural}.
Standard arguments (presented for completeness in \cref{sec:Tools}) show that if $\cm_{\cg_1} \neq \cm_{\cg_2}$ then the following apparently stronger condition holds:
\begin{equation}
\label{eqn:new-gen-id}
  \dim(\cm_{\cg_1}\cap \cm_{\cg_2}) < \max\left(\dim(\cm_{\cg_1}),\dim(\cm_{\cg_2})\right).
\end{equation}
Hence, the intersection of the two models is a Lebesgue null set in their union which means that a sample covariance matrix from the union can be assigned almost surely to one and only one of the two models under the assumption that the sample is generic with respect to the true model which generated it. Note that the definition in \cref{eqn:new-gen-id}
only differs from the standard definition of generic identifiability which is given by
\begin{equation}
  \dim(\cm_{\cg_1}\cap \cm_{\cg_2}) < \min \left(\dim(\cm_{\cg_1}),\dim(\cm_{\cg_2})\right)
\end{equation}
in the case that the models involved are \emph{nested}, meaning that $\cm_{\cg_1} \subseteq \cm_{\cg_2}$. This new definition is motivated by the fact that if the true data-generating distribution lies in the smaller model $\cm_{\cg_1} \subseteq \cm_{\cg_2}$, then model selection methods like the Lasso from \cite{dettling2022lasso} would be able to consistently infer that $\cm_{\cg_1}$ is the true model.
\end{remark}

Bayesian network models (and similarly Gaussian linear structural equation models) are completely defined by conditional independence constraints which are derived from the graph using the d-separation criterion \cite{lauritzen1996graphical}. Two graphs yield the same Bayesian network model if and only if the two graphs have the same set of d-separations. This characterization of Bayesian network models was instrumental to the well-known model equivalence results derived for them \cite{andersson1997characterization, VermaPearl}.  For GCLMs, the situation is very different. It was recently shown that all conditional independencies for GCLMs follow from \emph{marginal} independencies, i.e., zeros in the covariance matrix, which can be derived from the associated trek graph; cf.~\cite{drton2024conditional}. This shows that conditional independence is vastly insufficient for understanding model equivalence between GCLMs and thus makes the problem substantially more difficult than the classical case. For example, the forward and backward paths in \Cref{eg:123} both satisfy no marginal and therefore no conditional independencies and yet they are distinguishable from each other.

Thus, we shall take a slightly different perspective. By a celebrated result of Verma and Pearl~\cite{VermaPearl}, two DAGs on the same vertex set $V$ define the same Bayesian network model if and only if they have the same skeleton and the same v-structures. In other words, $\cg_1$ and $\cg_2$ are model equivalent if and only if the induced subgraphs $\cg_1[K]$ and $\cg_2[K]$ are model equivalent for all $K \subseteq V$ with $|K|=3$. When understood this way, our first main result is a direct analogue for Lyapunov DAG models.

\begin{theorem}
\label{thm:struct_id_induced_subgraph}
Let $\cg_1$ and $\cg_2$ be DAGs on vertex set~$V$. Then $\cg_1$ and $\cg_2$ are model equivalent if and only if they have the same skeleton and the induced subgraphs $\cg_1[K]$ and $\cg_2[K]$ are model equivalent for all $K \subseteq V$ with $|K|=4$.
\end{theorem}

This theorem reduces the problem of determining all model equivalence classes --- and thus all structurally identifiable graphs --- to characterizing the model equivalence classes of graphs with four nodes. A series of computations similar to that of \cref{eg:123} shows that every Lyapunov DAG model on $n \le 3$ vertices is structurally identifiable (within the class of Lyapunov DAG models) except for graphs which contain a $2$- or $3$-clique as a weakly connected component. For $n=4$ there are seven types of non-identifiable models given in \Cref{fig:NonIdent}.
This classification furnishes a graphical polynomial-time algorithm for checking model equivalence and is similar in spirit to the characterizations of many graph properties, such as planarity or chordality, in terms of \emph{forbidden minors}.

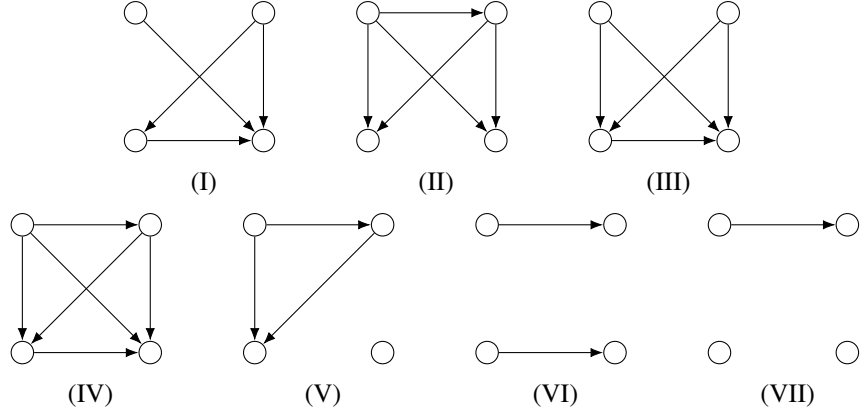
\begin{figure}
\renewcommand\thesubfigure{\Roman{subfigure}}%
\centering%
\begin{subfigure}[b]{0.15\textwidth}
\begin{tikzpicture}[every node/.style={draw, circle, inner sep=3pt}, scale=1.7, every edge/.style={draw, >=Latex}]
\node (1) at (0,0) {}; \node (2) at (1,0) {};
\node (3) at (0,-1) {}; \node (4) at (1,-1) {};
\draw (1) edge[->] (4); \draw (2) edge[->] (3);
\draw (2) edge[->] (4); \draw (3) edge[->] (4);
\end{tikzpicture}
\caption{} \label{fig:NonIdent1}
\end{subfigure}
\hspace{2.5em}
\begin{subfigure}[b]{0.15\textwidth}
\begin{tikzpicture}[every node/.style={draw, circle, inner sep=3pt}, scale=1.7, every edge/.style={draw, >=Latex}]
\node (1) at (0,0) {}; \node (2) at (1,0) {};
\node (3) at (0,-1) {}; \node (4) at (1,-1) {};
\draw (1) edge[->] (2); \draw (1) edge[->] (3); \draw (1) edge[->] (4);
\draw (2) edge[->] (3); \draw (2) edge[->] (4);
\end{tikzpicture}
\caption{} \label{fig:NonIdent2}
\end{subfigure}
\hspace{2.5em}
\begin{subfigure}[b]{0.15\textwidth}
\begin{tikzpicture}[every node/.style={draw, circle, inner sep=3pt}, scale=1.7, every edge/.style={draw, >=Latex}]
\node (1) at (0,0) {}; \node (2) at (1,0) {};
\node (3) at (0,-1) {}; \node (4) at (1,-1) {};
\draw (1) edge[->] (3); \draw (1) edge[->] (4);
\draw (2) edge[->] (3); \draw (2) edge[->] (4);
\draw (3) edge[->] (4);
\end{tikzpicture}
\caption{} \label{fig:NonIdent3}
\end{subfigure}
\hspace{2.5em} \\
\begin{subfigure}[b]{0.15\textwidth}
\begin{tikzpicture}[every node/.style={draw, circle, inner sep=3pt}, scale=1.7, every edge/.style={draw, >=Latex}]
\node (1) at (0,0) {}; \node (2) at (1,0) {};
\node (3) at (0,-1) {}; \node (4) at (1,-1) {};
\draw (1) edge[->] (2); \draw (1) edge[->] (3); \draw (1) edge[->] (4);
\draw (2) edge[->] (3); \draw (2) edge[->] (4); \draw (3) edge[->] (4);
\end{tikzpicture}
\caption{} \label{fig:NonIdent4}
\end{subfigure}
\hspace{2.5em}
\begin{subfigure}[b]{0.15\textwidth}
\begin{tikzpicture}[every node/.style={draw, circle, inner sep=3pt}, scale=1.7, every edge/.style={draw, >=Latex}]
\node (1) at (0,0) {}; \node (2) at (1,0) {};
\node (3) at (0,-1) {}; \node (4) at (1,-1) {};
\draw (1) edge[->] (2); \draw (1) edge[->] (3); \draw (2) edge[->] (3);
\end{tikzpicture}
\caption{} \label{fig:NonIdent5}
\end{subfigure}
\hspace{2.5em}
\begin{subfigure}[b]{0.15\textwidth}
\begin{tikzpicture}[every node/.style={draw, circle, inner sep=3pt}, scale=1.7, every edge/.style={draw, >=Latex}]
\node (1) at (0,0) {}; \node (2) at (1,0) {};
\node (3) at (0,-1) {}; \node (4) at (1,-1) {};
\draw (1) edge[->] (2); \draw (3) edge[->] (4);
\end{tikzpicture}
\caption{} \label{fig:NonIdent6}
\end{subfigure}
\hspace{2.5em}
\begin{subfigure}[b]{0.15\textwidth}
\begin{tikzpicture}[every node/.style={draw, circle, inner sep=3pt}, scale=1.7, every edge/.style={draw, >=Latex}]
\node (1) at (0,0) {}; \node (2) at (1,0) {};
\node (3) at (0,-1) {}; \node (4) at (1,-1) {};
\draw (1) edge[->] (2);
\end{tikzpicture}
\caption{} \label{fig:NonIdent7}
\end{subfigure}
\caption{The seven fundamental types of non-identifiable Lyapunov DAG models on four vertices. The classes (\subref{fig:NonIdent4})--(\subref{fig:NonIdent7}) are trivial in that they consist of graphs whose connected components are cliques. The other equivalence classes are made up as follows: (\subref{fig:NonIdent1}) Twelve classes with two DAGs each; (\subref{fig:NonIdent2}) six classes with two DAGs each; and (\subref{fig:NonIdent3}) six classes with two DAGs each.}
\label{fig:NonIdent}
\end{figure}

\begin{theorem}
\label{thm:struct-id-forbidden-minors}
A DAG $\cg = (V, E)$ is structurally identifiable within the class of DAGs if and only if for each non-loop edge $i \to j \in E$ there exist $k, l \in V$ such that $i, j, k, l$ are all distinct and $\cg[i,j,k,l]$ is not isomorphic to any of the graphs in \Cref{fig:NonIdent}.
\end{theorem}

\begin{corollary}
It is decidable in polynomial time whether a DAG is structurally identifiable within the class of Lyapunov DAG models. It is also decidable in polynomial time whether two DAGs are model equivalent.
\end{corollary}

We also provide a \emph{transformational characterization} of model equivalence in the Lyapunov setting reminiscent of Chickering's ``covered edge flips''~\cite{chickering1995transformational}.

\begin{definition}\label{defn:super-covered-edge}
Let $\cg = (V, E)$ be a simple DAG. Then an edge $i \to j \in E$ is called \emph{super-covered} if it satisfies the following conditions:
\begin{itemize}
\item $\ch(i) = \ch(j) \cup \{j\}$ and $\pa(i) \cup \{i\} = \pa(j)$,
\item for each $k \in \pa(j)$ and $l \in \ch(i)$ it holds that $k \to l \in E$, and
\item for all $k \in V \setminus \{i, j\}$ either $k \in \neighbors(i) \cap \neighbors(j)$ or $k \indep \{i, j\}$.
\end{itemize}
\end{definition}

The marginal independence $k \indep \{i,j\}$ is equivalent, by \cite[Proposition~2.3]{hansen2025trek}, to the non-existence of treks from $k$ to either $i$ or~$j$. The definition is therefore entirely graphical.

\begin{theorem}\label{theorem:super covered edge flip}
Let $\cg_1$ and $\cg_2$ be two simple DAGs. Then $\cm_{\cg_1}=\cm_{\cg_2}$ if and only if there exists a sequence of super-covered edge flips that transforms $\cg_1$ to $\cg_2$. Moreover, if $\cg_1$ and $\cg_2$ are model equivalent (hence have the same skeleton) and $\delta = |E(\cg_1) \setminus E(\cg_2)|$ is the number of differently oriented edges between these two graphs, then it takes only $\delta$ super-covered edge flips to transform $\cg_1$ into $\cg_2$.
\end{theorem}

\begin{example}
Let $\cg_1$ and $\cg_3$ be two simple DAGs as shown in \cref{figure:flipsupercovered}. Computing the corresponding $f^{ij}_{\cg_1}$ and $f^{ij}_{\cg_3}$ gives us that $\cm_{\cg_1}=\cm_{\cg_3}.$ Now, observe that the edge $2\rightarrow 3$ is super-covered in $\cg_1$ as $\pa(2)\cup \{2\}=\pa(3)=\{1,2\}$, $\ch(2)=\ch(3)\cup \{3\}=\{ 4,5 \}$, and every parent of $3$ is a parent of every child of $2$; the vertex $6$, which is not a neighbor of $\{2,3\}$, is marginally independent of $\{2,3\}$. Thus, by \cref{theorem:super covered edge flip}, we can flip the edge $2\rightarrow 3$ to reach $\cg_2$ without changing the model. In a similar way, it can be seen that the edge $1\rightarrow 3$ is now super-covered in $\cg_2$, although it was not super-covered in $\cg_1$. Flipping $1\rightarrow 3$ allows us to transform $\cg_2$ into $\cg_3$, and also keeps the model intact.

\begin{figure}
\begin{tikzpicture}[scale=0.9, every node/.style={draw, circle, inner sep=2pt}, every edge/.style={draw, >=Latex}]
  \node (2) at (0,1)  {$2$};
  \node (4) at (-1.5,-1) {$4$};
  \node (1) at (1,0) {$1$};
  \node (3) at (0,-1) {$3$};
  \node (6) at (-2.5,0) {$6$};
  \node (5) at (-1.5,1) {$5$};
  \draw (1) edge[->,black] (2) ;
  \draw (1) edge[->,black] (5) ;
  \draw (1) edge[->,black] (4) ;
  \draw (2) edge[->,black] (5) ;
  \draw (3) edge[->,black] (5) ;
  \draw (6) edge[->,black] (4) ;
  \draw (1) edge[->,black] (3) ;
  \draw (2) edge[->,black] (4) ;
  \draw (3) edge[->,black] (4) ;
  \draw (2) edge[->,red,thick] (3) ;
\node[draw=none] () at (-.55,-2) {\Large$\cg_1$};
\node[draw=none] () at (1.9,0) {\large$\longrightarrow$};

\node (12) at (2.8,0)  {$6$};
\node (11) at (3.8,1)  {$5$};
\node (8) at (5.3,1)  {$2$};
\node (10) at (3.8,-1) {$4$};
\node (7) at (6.3,0) {$1$};
\node (9) at (5.3,-1) {$3$};
  \draw (7) edge[->,black] (8) ;
  \draw (7) edge[->,blue,thick] (9) ;
  \draw (7) edge[->,black] (10) ;
  \draw (7) edge[->,black] (11) ;
  \draw (9) edge[->,red,thick] (8) ;
  \draw (8) edge[->,black] (10) ;
  \draw (8) edge[->,black] (11) ;
  \draw (9) edge[->,black] (10) ;
  \draw (9) edge[->,black] (11) ;
  \draw (12) edge[->,black] (10) ;
\node[draw=none] () at (4.7,-2) {\Large$\cg_2$};
\node[draw=none] () at (7.25,0) {\large$\longrightarrow$};

  \node (18) at (8.3,0)  {$6$};
  \node (17) at (9.3,1)  {$5$};
  \node (14) at (10.8,1)  {$2$};
  \node (16) at (9.3,-1) {$4$};
  \node (13) at (11.8,0) {$1$};
  \node (15) at (10.8,-1) {$3$};
  \draw (13) edge[->,black] (14) ;
  \draw (15) edge[->,blue,thick] (13) ;
  \draw (13) edge[->,black] (16) ;
  \draw (13) edge[->,black] (17) ;
  \draw (15) edge[->,black] (14) ;
  \draw (14) edge[->,black] (16) ;
  \draw (14) edge[->,black] (17) ;
  \draw (15) edge[->,black] (16) ;
  \draw (15) edge[->,black] (17) ;
  \draw (18) edge[->,black] (16) ;
\node[draw=none] () at (10.2,-2) {\Large$\cg_3$};
\end{tikzpicture}
\caption{Connecting model equivalent graphs via two super-covered edge flips.}
\label{figure:flipsupercovered}
\end{figure}
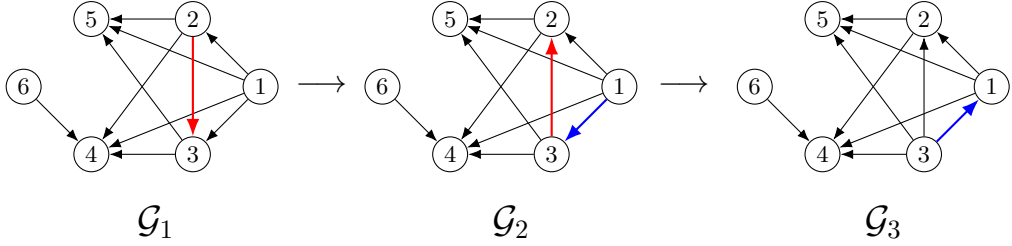
\end{example}

Recall that in Chickering's transformational characterization of model equivalence \cite{chickering1995transformational} an edge $i \to j \in E$ is \emph{covered} in $\cg$ if $\pa(i) \cup \{i\} = \pa(j)$. Thus any super-covered edge is also covered in the classical sense which leads to the following corollary.

\begin{corollary}
\label{cor:lyap-refines-bayes}
The model equivalence classes obtained for GCLMs are a refinement of the Markov equivalence classes of Bayesian network models. In other words, if two DAGs are not Markov-equivalent (as Bayesian networks) then they are not model equivalent as GCLMs.
\end{corollary}

Another important feature of \cref{defn:super-covered-edge} is that the neighborhood of a super-covered edge must be quite dense. For instance, it is easy to see that the only polytree which may contain a super-covered edge is the 2-clique $1 \to 2$.

\begin{corollary}
Let $\ct = (V, E)$ be a polytree on $|V| \ge 3$ nodes. Then $\ct$ is structurally identifiable within the class of all Lyapunov DAG models.
\end{corollary}

The precise connection between our characterization of model equivalence via induced subgraphs \cref{thm:struct_id_induced_subgraph} and the transformational characterization in \cref{theorem:super covered edge flip} may not be immediately obvious, however \cref{thm:struct-id-forbidden-minors} provides a clearer link. In particular, note that every graph in \cref{fig:NonIdent} contains at least one super-covered edge, and thus there is at least one other graph which is model equivalent to it. On the other hand, an edge $i \to j \in E$ is super-covered if and only if it is super-covered in the induced subgraph $\cg[i,j,k,l]$ for all $k, l \in V \setminus \{i, j\}$.

\section{Proofs}
\label{sec:proofs}

This section contains the proofs of our characterizations of model equivalence for Lyapunov DAG models via induced subgraphs (\cref{thm:struct_id_induced_subgraph}) and via edge flips (\cref{theorem:super covered edge flip}). While these results appear separate, their proofs are intertwined. After developing general tools in \cref{sec:Tools}, we first show that model equivalent graphs have the same skeleton and model equivalent subgraphs. Next, we show that flipping super-covered edges preserves the model. The proof of \cref{thm:struct_id_induced_subgraph} is finished by tying the existence of super-covered edges in a graph to their existence in induced subgraphs on four nodes. From there the completeness of super-covered edge flips follows by a graph-theoretic argument.

\subsection{Techniques for Proving Structural Identifiability Results}
\label{sec:Tools}

Lyapunov models are \emph{algebraic} subsets of the cone of positive definite matrices, i.e., they are defined by polynomial equations. This was shown in \cite[Corollary~5.4]{dettling2022identifiability} but also follows directly from \cref{prop:MissingEdge}. Moreover GCLMs are \emph{irreducible} as algebraic sets. This means that $\cm_\cg$ cannot be written as the union of two proper algebraic subsets. This follows from elementary topological considerations: since the set of stable matrices $M \in \rr^E$ is irreducible and the parametrization $M \mapsto \Sigma$ via~\eqref{eqn:BM} is continuous, the image $\cm_\cg$ must be irreducible as well. We make use of this property in two essential~ways:

\begin{proposition} \label{prop:Geom}
Let $\cm$ and $\cm'$ be two algebraic sets.
\begin{enumerate}
\item If $\cm \subseteq \cm'$ are both irreducible and have the same dimension then they are equal.
\item If $\cm$ is irreducible and not contained in $\cm'$ then $\dim(\cm \cap \cm') < \dim(\cm)$.
\end{enumerate}
\end{proposition}

\begin{proof}
The first claim is a direct consequence of the definition of dimension of an algebraic set as the length of the longest chain of irreducible algebraic subsets. Such definition coincides with the usual notion of dimension for manifolds; cf.~\cite[Section~9.6]{CLO}.
The second assertion follows from the first one by the decomposition of an algebraic set into irreducible sets; cf.~\cite[Section~4.2]{sullivant2018algebraic}.
\end{proof}

Recall from \cite[Theorem~7.1]{dettling2022identifiability} that if $\cg$ is a simple graph then $\cm_{\cg}$ is globally identifiable. In this case $\dim(\cm_\cg) = |E(\cg)|$ is equal to the number of parameters in~$M$ which are not set to zero. This holds in particular for DAGs.
Non-simple graphs fail to be globally identifiable but may still satisfy a weaker condition known as \emph{generic (parameter) identifiability}; cf.~\cite[Section~8]{dettling2022identifiability}. This means that the parameter matrix $M$ can be uniquely identified for almost all distributions in the model. The subset where unique identifiability fails is a lower-dimensional algebraic set and referred to as the \emph{locus of non-identifiability}; cf.~\cite[Remark~3.2 \& Lemma~4.3]{dettling2022identifiability}.

The following lemmas show how parameter identifiability can be used to obtain both positive and negative structural identifiability results. %

\begin{lemma}
\label{lem:lyap_struct_id_from_param_id}
Let $\cg_1 = (V, E_1)$ and $\cg_2 = (V, E_2)$ be distinct simple graphs with common vertex set $V$. Suppose there exists a common extension $\cg = (V, E)$, i.e., $E_1 \cup E_2 \subseteq E$ which is generically parameter identifiable and such that $\cm_{\cg_1}$ and $\cm_{\cg_2}$ are not contained in the locus of non-identifiability. Then $\cm_{\cg_1}$ and $\cm_{\cg_2}$ are distinguishable.
\end{lemma}

\begin{proof}
Since $E_1 \neq E_2$, their symmetric difference is non-empty. By relabeling the graphs, we may assume that there is $i\to j \in E_1 \setminus E_2$. The assumptions and \cref{prop:Geom} imply that the locus of non-identifiability of $\cm_{\cg}$ intersects $\cm_{\cg_1} \subseteq \cm_{\cg}$ in a lower-dimensional subset. Hence, we may pick $\Sigma \in \cm_{\cg_1}$ whose parameter matrix $M \in \rr^E$ is identifiable and has $m_{ji} \not= 0$. But then $\Sigma$ violates the missing-edge relation for $i\to j \in E \setminus E_2$ which shows that $\Sigma \notin \cm_{\cg_2}$.
\end{proof}

\begin{lemma} \label{lemma:NonParam}
Let $\cg_1, \cg_2, \cg$ be graphs such that $E_1, E_2 \subseteq E$ and $|E \setminus E_i| = 1$. If $\cg$ is not parameter identifiable but $\cg_1, \cg_2$ are parameter identifiable then $\cg_1$ and $\cg_2$ are model equivalent.
\end{lemma}

\begin{proof}
The model $\cm_\cg$ is not parameter identifiable which means $\dim(\cm_\cg) < |E|$. On~the other hand, it contains the two models $\cm_{\cg_1}$ and $\cm_{\cg_2}$ which, because they are identifiable, have dimension $|E_1| = |E_2| = |E|-1$. Hence,
\[
\dim(\cm_{\cg_1}) = \dim(\cm_\cg) = \dim(\cm_{\cg_2}).
\]
Now $\cm_{\cg_1} = \cm_{\cg} = \cm_{\cg_2}$ follows from \cref{prop:Geom}.
\end{proof}

\begin{example}
An example of a non-identifiable model is given by the non-simple graph
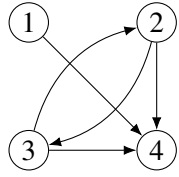
\begin{center}
\begin{tikzpicture}[every node/.style={draw, circle, inner sep=2pt}, scale=1.7, every edge/.style={draw, >=Latex}]
\node (1) at (0,0) {1}; \node (2) at (1,0) {2};
\node (3) at (0,-1) {3}; \node (4) at (1,-1) {4};
\draw (1) edge[->] (4); \draw (2) edge[->, bend left=30] (3); \draw (3) edge[->, bend left=30] (2);
\draw (2) edge[->] (4); \draw (3) edge[->] (4);
\end{tikzpicture}
\end{center}
This corresponds to type \cref{fig:NonIdent}~(\subref{fig:NonIdent1}) and was also discussed in \cite[Example~8.4]{dettling2022identifiability}. By removing one of the two edges between $2$ and $3$, we obtain simple graphs to which \cref{lemma:NonParam} applies. It shows that all three graphs have the same model which is simply described by the marginal independence $X_1 \indep (X_2, X_3)$.
\end{example}

Our last result in this subsection shows how model equivalence results for general graphs may be reduced to \emph{almost-complete graphs}, i.e., simple graphs with one missing edge.

\begin{lemma}
\label{lemma:model-via-almost-complete-intersection}
Let $\cg$ be a simple graph with completion~$\cg'$ and missing edge set $\tilde{E} = E(\cg') \setminus E(\cg)$. For each missing edge $e \in \tilde{E}$ let $\cg_e = \cg' \setminus e$ be the completion with that edge removed.
Then $\cm_\cg = \bigcap_{e \in \tilde{E}} \cm_{\cg_e}$.
\end{lemma}

\begin{proof}
Observe that $\cg'$ is a completion of each of the graphs~$\cg_e$. By \cref{prop:MissingEdge} the model $\cm_{\cg}$ is defined by the conditions $f_{\cg'}^{e}(\Sigma) = 0$ for all $e \in \tilde{E}$. Each of these equations defines one of the models $\cm_{\cg_e}$, again by \cref{prop:MissingEdge}. Hence $\cm_{\cg}$ and $\bigcap_{e \in \tilde{E}} \cm_{\cg_e}$ are defined by the same set of equations and must therefore be equal.
\end{proof}

\subsection{Distinguishing Models via Induced Subgraphs} \label{sec:Subgraphs}

In this section, we prove the necessity of the induced subgraph criterion in \cref{thm:struct_id_induced_subgraph}. If $\cg_1$ and $\cg_2$ are two model equivalent DAGs, then they must have the same skeleton and model equivalent induced subgraphs.

\begin{lemma}
\label{lemma:skel_struct_id}
If two simple graphs $\cg_1$ and $\cg_2$ have different skeletons then $\cm_{\cg_1} \neq \cm_{\cg_2}$.
\end{lemma}

\begin{proof}
Let $i\to j$ be an edge in $\cg_1$ such that neither $i\to j$ nor $j\to i$ are edges in~$\cg_2$. Consider the subgraph $\cg_1' = (V, \{i\to j\})$ of $\cg_1$ in which only the edge $i\to j$ is retained. Since $i\to j, j\to i \notin E_2$, there exists a complete simple graph $\cg$ such that $E_2 \cup \{i\to j\} \subseteq E(\cg)$, thus $\cg_1'$ and $\cg_2$ are both subgraphs of~$\cg$. Since $\cg$ is simple, it is parameter identifiable and \cref{lem:lyap_struct_id_from_param_id} implies that $\cg_1'$ and $\cg_2$ are distinguishable. This yields a point $\Sigma \in \cm_{\cg_1'} \subseteq \cm_{\cg_1}$ which is not in $\cm_{\cg_2}$ which furthermore shows that $\cg_1$ and $\cg_2$ are distinguishable.
\end{proof}

We now proceed with the proof that model equivalent graphs have model equivalent induced subgraphs. In other words, if two graphs contain distinguishable subgraphs, they must be distinguishable as well. We need the following lemma about graphs with isolated vertices.

\begin{lemma} \label{lem:induced_subgraph_linear}
Let $\cg = (V, E)$ be a simple graph and $i \in V$. Let $\cg' = (V, E')$ be the graph obtained from $\cg$ by deleting all edges into and out of~$i$. Then $\Sigma \in \cm_{\cg'}$ if and only if $\Sigma \in \cm_{\cg}$ and $\sigma_{ij} = 0$ for $j \neq i$.
\end{lemma}

\begin{proof}
The inclusion $E' \subseteq E$ directly yields $\cm_{\cg'} \subseteq \cm_{\cg}$. Since $i$ is isolated in $\cg'$ we additionally have $\sigma_{ij} = 0$ for all $j \neq i$ by \cite[Proposition~8.3]{dettling2022identifiability}. Note that the assumption $C = 2 \Id_\nrnodes$ is required to apply this result.
For the other inclusion, consider any $\Sigma \in \cm_\cg$ with $\sigma_{ij} = 0$ for all $j \neq i$. It satisfies the Lyapunov equation $M \Sigma + \Sigma M^T = -C$ for a unique stable matrix~$M \in \rr^{E}$. If $j \neq i$ is not a parent of $i$ in $\cg$, then by definition of the model $m_{ij} = 0$. For~any $j \in \pa_\cg(i)$, the Lyapunov equation in the $(i,j)$ entry~reads
\[
  (m_{ii} + m_{jj}) \sigma_{ij} + \sum_{k \in \pa_\cg(i)} m_{ik} \sigma_{kj} + \sum_{k \in \pa_\cg(j)} m_{jk} \sigma_{ik} = 0.
\]
By assumption, every term in the sum over $k \in \pa_\cg(j)$ is zero and $\sigma_{ij}=0$ as well. This implies $\sum_{k \in \pa_\cg(i)} m_{ik} \sigma_{kj} = 0$. As $j$ ranges in $\pa_\cg(i)$, these equations can be restated as
\[
  M_{i,\pa_\cg(i)} \Sigma_{\pa_\cg(i)} = 0,
\]
and positive definiteness of $\Sigma$ implies that all $m_{ij} = 0$ for $j \in \pa_\cg(i)$.
We now consider the edge coefficients $m_{ji}$. Again, if $j$ is not a child of $i$ in $\cg$, this coefficient must be zero. So assume that $i \in \pa_\cg(j)$ and consider the $(j,i)$ entry of the Lyapunov equation:
\[
  (m_{ii} + m_{jj}) \sigma_{ij} + \sum_{k \in \pa_\cg(j)} m_{jk} \sigma_{ik} + \sum_{k \in \pa_\cg(i)} m_{ik} \sigma_{jk} = 0.
\]
Since $\sigma_{ik} = 0$ and $m_{ik} = 0$ for all $k \neq i$ as shown above, this equation simplifies to $m_{ji} \sigma_{ii} = 0$ which implies $m_{ji} = 0$. But then $M \in \rr^{E'}$ and thus $\Sigma \in \cm_{\cg'}$.
\end{proof}

\begin{proposition} \label{prop:induced_subgraph_distinguish}
Let $\cg_1 = (V, E_1)$ and $\cg_2 = (V, E_2)$ be simple graphs with the same number of edges. If there exists a set $K \subseteq V$ such that the induced subgraphs $\cg_1[K]$ and $\cg_2[K]$ are distinguishable, then $\cg_1$ and $\cg_2$ are distinguishable as well.
\end{proposition}

\begin{proof}
Apply \cref{lem:induced_subgraph_linear} inductively to isolate all vertices $V \setminus K$ in $\cg_1$. The resulting graph $\cg_1'$ has as its model the intersection of $\cm_{\cg_1}$ with the linear space $W_{K}$ defined by the conditions $\sigma_{ij} = 0$ for $i \in V \setminus K$ and $j \neq i$. This linear space depends only on $K$ and not on the graph. All matrices in $\cm_{\cg_1'}$ are block-diagonal with a non-trivial $K \times K$ block and an arbitrary diagonal matrix in its complement. The block matrices $\{ \Sigma_K : \Sigma \in \cm_{\cg_1'} \}$ form the model of~$\cg_1[K]$. Thus the assumption of our claim is equivalent to $\cm_{\cg_1} \cap W_{K} \neq \cm_{\cg_2} \cap W_{K}$ which readily implies $\cm_{\cg_1} \neq \cm_{\cg_2}$.
\end{proof}

Combining the previous results we obtain the necessity of the condition in \Cref{thm:struct_id_induced_subgraph}.

\begin{corollary}\label{cor:only if direction}
Let $\cg_1=(V,E_1)$ and $\cg_2=(V,E_2)$ be two simple graphs. If $\cg_1$ and $\cg_2$ are model equivalent, then $\cg_1$ and $\cg_2$ have the same skeleton and the induced subgraphs $\cg_1[K]$ and $\cg_2[K]$ are model equivalent for all $K\subseteq V$.
\end{corollary}

\subsection{Equivalence Under Super-Covered Edge Flips} \label{sec:EdgeFlips}

We now shift our focus to proving model equivalence among different graphs. For a covariance matrix $\Sigma \in \PD_V$ and a subset $K \subseteq V$ we denote by $\Sigma_K$ the covariance matrix of the subset $K$, i.e., the submatrix of $\Sigma$ whose rows and columns are indexed by~$K$.

\begin{lemma} \label{lemma: adding a sink}
Let $\cg = (V, E)$ be any simple graph and consider the graph $\cg_1 = (V_1, E_1)$ obtained from $\cg$ by adding a new vertex $w$ with its self-loop $w \to w$ as well as the edges $v \to w$ for all $v\in V$. Then for any $\Sigma \in \PD_{V_1}$ we have $\Sigma \in \cm_{\cg_1}$ if and only if $\Sigma_V \in \cm_\cg$.
\end{lemma}

\begin{proof}
Let $\cg_1'$ be a completion of~$\cg_1$. This implies that $\cg' \defas \cg_1' \setminus w$ is a completion of~$\cg$. By construction $\cg$ and $\cg_1$ have the same set of missing edges with respect to these two completions. By virtue of \Cref{prop:MissingEdge}, $\Sigma \in \PD_{V_1}$ is in $\cm_{\cg_1}$ if and only if it satisfies all missing-edge relations for $\cg_1$ with respect to $\cg_1'$. To compute these relations, we rearrange the rows and columns of $A_{\cg_1'}(\Sigma)$ to create a block structure. Order the rows into three groups: first $(k, l)$ with $k, l \in V$, then $(k, w)$ with $k \in V$ and lastly $(w,w)$. For the columns, choose an analogous ordering of the edges with first $i \to j$ for $i, j \in V$, then $i \to w$ and lastly $w \to w$. \Cref{def:Asigma} then yields the following block structure:
\begin{align}\label{eq:blocks}
  A_{\cg_1'}(\Sigma) = \begin{pmatrix}
    A_{\cg'}(\Sigma_V) & 0 & 0 \\
    \ast & \Sigma_V & \sigma_{kw} \\
    \ast & 2\sigma_{wi} & 2\sigma_{ww}
  \end{pmatrix}.
\end{align}
Note that the four bottom right blocks of \eqref{eq:blocks} form $\Sigma$ with its last row doubled. Hence, $f^{ij}_{\cg_1'}(\Sigma) = 2 \det(\Sigma) \, f^{ij}_{\cg'}(\Sigma_V)$ and thus by invoking \Cref{prop:MissingEdge} once more for the model $\cm_{\cg}$ and using $2 \det(\Sigma) > 0$, it follows that $\Sigma \in \cm_{\cg_1}$ if and only if $\Sigma_V \in \cm_{\cg}$.
\end{proof}

\begin{lemma} \label{lem:remove full sink clique}
Let $\cg_1 = (V, E_1)$ and $\cg_2 = (V, E_2)$ be two DAGs with the same number of edges. Suppose that $V = P \cup C$ where $P \cap C = \emptyset$, $C$ is a clique in both $\cg_1$ and $\cg_2$, and ${p \to c} \in E_1 \cap E_2$ for all $p \in P$ and all $c \in C$. Then $\cm_{\cg_1[P]} = \cm_{\cg_2[P]}$ implies $\cm_{\cg_1} = \cm_{\cg_2}$.
\end{lemma}

\begin{proof}
The idea is that none of the vertices in $C$ appear in any missing edge relation of $\cg_1$ or $\cg_2$ with respect to any of their completions, so they contribute no information to the model distinguishability problem. To make this rigorous, observe that iterated use of \cref{lemma: adding a sink} shows that $\Sigma \in \PD_V$ is in $\cm_{\cg_1}$ if and only if $\Sigma_P \in \cm_{\cg_1[P]}$; analogously $\Sigma \in \cm_{\cg_2}$ if and only if $\Sigma_P \in \cm_{\cg_2[P]}$. Hence if $\cm_{\cg_1[P]} = \cm_{\cg_2[P]}$, it follows that $\cm_{\cg_1} = \cm_{\cg_2}$.
\end{proof}

\begin{lemma} \label{lemma: almost complete}
Let $\cg_1$ be an almost-complete DAG, i.e., a DAG with $\binom{n+1}{2}-1$ edges, and suppose that $a \to b$ is a super-covered edge in $\cg_1$. Let $\cg_2$ be the graph obtained by flipping $a \to b$. Then $\cg_2$ is a DAG and $\cm_{\cg_1} = \cm_{\cg_2}$.
\end{lemma}

\begin{proof}
Since $\cg_1$ is almost-complete and contains a super-covered edge $a \to b$, its vertex set decomposes as $V = P \cup \{a,b\} \cup C$ where $P = \pa_{\cg_1}(a)$ and $C = \ch_{\cg_1}(b)$. It is easy to see that $\cg_2$ with the flipped edge is still a DAG. We now distinguish cases depending on what kind of edge is missing to the completion of~$\cg_1$. Since $a \to b$ is super-covered, there is an edge from every parent of $b$ to every child of $a$. This only leaves two possibilities for the missing~edge.

(1) Suppose that $p_1 \to p_2$ is missing in $\cg_1$ for $p_1, p_2 \in P$. In this case $\{a,b\} \cup C$ is a clique in both $\cg_1$ and $\cg_2$ which has no outgoing edges but an incoming edge from every $p \in P$. Since $\cg_1[P] = \cg_2[P]$, \cref{lem:remove full sink clique} shows $\cm_{\cg_1} = \cm_{\cg_2}$.

(2) Suppose that the missing edge is $c_1 \to c_2$ for $c_1, c_2 \in C$. Consider the non-simple graph $\cg = \cg_1 \cup \cg_2$ which has $\binom{n+1}{2}$ edges. We show that this graph is not parameter identifiable so that \cref{lemma:NonParam} yields the desired conclusion. To accomplish this, we pick any $\Sigma \in \cm_{\cg}$ and show that the associated matrix $A_\cg(\Sigma)$ is singular by exhibiting a non-zero element in~its~kernel. By \cite[Lemma~4.3]{dettling2022identifiability} this yields non-identifiability.

Recall from \cref{def:Asigma} that $A_\cg(\Sigma)$ is a submatrix of the full $A(\Sigma)$ matrix in which the columns which do not correspond to edges of $\cg$ are omitted. A basis of the kernel of $A(\Sigma)$ was given in \cite[Lemma~6.5 \& Eq.~(6.1)]{dettling2022identifiability}: it consists of the columns of the matrix~$H(\Sigma)$~below. This matrix has its rows indexed by all possible edges $i \to j$, $i, j \in V$, and its columns by pairs $(k, l)$ for $k, l \in V$ with $k < l$ in the topological ordering of~$\cg_1$. Its entries are given by
\begin{equation}
  H(\Sigma)_{i\to j, (k,l)} = \begin{cases}
    -\sigma_{lj}, & \text{if $i=k$}, \\
     \sigma_{kj}, & \text{if $i=l$}, \\
     0, & \text{else}.
  \end{cases}
\end{equation}
Let $Q = P \cup \{a,b\}$ and $T = \{ (s,t) \in Q \times Q : s < t\}$ where the ordering is that in $\cg_1$ where $p < a < b$ for all $p \in P$. We may also assume that $Q = \{1, \dots, |P|+2\}$. Then define the vector~$d \in \rr^{\binom{n}{2}}$ via
\begin{equation}
  d_{(s,t)} = \begin{cases}
    (-1)^{s+t+1} \det\left(\Sigma_{P,Q \setminus \{s,t\}}\right), & \text{if $(s,t) \in T$}, \\
    0, & \text{else}.
  \end{cases}
\end{equation}
By construction $D = H(\Sigma) d_{(s,t)}$ is a vector in the kernel of~$A(\Sigma)$. Our goal is now to prove that $D_{i \to j} = 0$ whenever $i \to j \notin E(\cg)$. In this case, the subvector $D_{E(\cg)}$ must be in the kernel of the submatrix $A_\cg(\Sigma)$.

Consider any possible edge $i \to j$, with $i, j \in V$, and write out the corresponding entry
\begin{align}
  D_{i \to j} &= \sum_{(s,t) \in T} d_{(s,t)} H(\Sigma)_{i\to j,(s,t)} \notag \\
  &= \sum_{(s,i) \in T} (-1)^{1+i+s} \sigma_{js} \det\left( \Sigma_{P,Q\setminus\{s,i\}} \right) + \sum_{(i,t) \in T} (-1)^{1+i+t} \sigma_{jt} \det\left( \Sigma_{P,Q\setminus\{t,i\}} \right). \label{eqn:Dij}
\end{align}
If $i \in C$ then both sums are empty by the definitions of~$T$ and $H(\Sigma)$, so this expression is zero. If $i \in Q$ and $j \in C$, then $i \to j \in E(\cg)$ and there is nothing to check. Now assume that $i, j \in Q$ and that $i \to j \notin E(\cg)$. In particular this means $j \in P$. The following identity is easy to verify following \eqref{eqn:Dij} and using Laplace expansion of the determinant on the right-hand side along the $j$-th row:
\begin{equation}
  \label{eqn:DijDet}
  D_{i \to j} = (-1)^{i+j}\det\left( \Sigma_{P \sqcup \{j\}, Q \setminus \{i\}} \right).
\end{equation}
Since $j \in P$, this matrix has a repeated row and thus its determinant vanishes, as required. Note also that $D_{E(\cg)} \in \ker(A_\cg(\Sigma))$ is non-zero since $D_{a\to b} = \pm\det\left( \Sigma_{P\cup \{b\}} \right) \neq 0$.
\end{proof}

\begin{proposition} \label{prop: edge flip}
Let $\cg_1 = (V, E_1) $ be a DAG with a super-covered edge $a \to b \in E_1$. If $\cg_2$ is obtained from $\cg_1$ by flipping $a \to b$ then $\cg_2$ is a DAG and $\cm_{\cg_1} = \cm_{\cg_2}$.
\end{proposition}

\begin{proof}
Since $a \to b$ is super-covered, the vertices partition as $V = P \cup \{a,b\} \cup C \cup M$ where $P = \pa_{\cg_1}(a)$, $C = \ch_{\cg_1}(b)$, and for each $m \in M$ we have $m \indep \{a,b\}$.
Observe that there is a topological ordering $\tau$ of $\cg_1$ in which $p <_\tau a <_\tau b <_\tau m <_\tau c$ for all $p \in P$, $m \in M$ and $c \in C$.
It follows that $\cg_2$ is a DAG with the same topological ordering except that $a$ and $b$ must be transposed. Extend $\cg_1$ to a complete DAG $\cg_1'$ by adding all missing edges of the form $i \to j$ where $i <_\tau j$. Add the same edges to $\cg_2$ to obtain a complete DAG $\cg_2'$. The edge $a \to b$ is still super-covered in $\cg_1'$ where the vertices in $M$ are no longer marginal to $\{a,b\}$ but are instead the children of $P \cup \{a,b\}$, and flipping it produces~$\cg_2'$.

Note that $\cg_1$ and $\cg_2$ have the same sets of missing edges with respect to $\cg_1'$ and $\cg_2'$, respectively. Denote this set by $\tilde E$. For any $e \in \tilde E$ the two graphs $\cg_1' \setminus e$ and $\cg_2' \setminus e$ are almost-complete and differ by a super-covered edge flip. Thus their models are the same by \cref{lemma: almost complete}, and $\cm_{\cg_1} = \bigcap_{e \in \tilde E} \cm_{\cg_1' \setminus e} = \bigcap_{e \in \tilde E} \cm_{\cg_2' \setminus e} = \cm_{\cg_2}$ follows by \cref{lemma:model-via-almost-complete-intersection}.
\end{proof}

\begin{example}
\begin{figure}
\begin{tikzpicture}[every node/.style={draw, circle, inner sep=2pt}, scale=1.2, every edge/.style={draw, >=Latex}]
\node (1) at (0,0) {1};
\node (2) at (-.5,1) {2};
\node (3) at (0,2) {3};
\node (4) at (1,0) {4};
\node (5) at (1,2) {5};
\node (6) at (2,1) {6};

\draw (1) edge[->] (2);
\draw (1) edge[->] (4);
\draw (1) edge[->] (5);
\draw (1) edge[->] (6);
\draw (2) edge[->] (3);
\draw (2) edge[->] (4);
\draw (2) edge[->] (5);
\draw (2) edge[->] (6);
\draw (3) edge[->] (4);
\draw (3) edge[->] (5);
\draw (3) edge[->] (6);
\draw (4) edge[red,thick,->] (5);
\draw (4) edge[->] (6);
\draw (5) edge[->] (6);

\node[draw=none] () at (0.5,-.55) {\Large$\cg_1$};

\node (7) at (4,0) {1};
\node (8) at (3.5,1) {2};
\node (9) at (4,2) {3};
\node (10) at (5,0) {4};
\node (11) at (5,2) {5};
\node (12) at (6,1) {6};

\draw (7) edge[->] (8);
\draw (7) edge[->] (10);
\draw (7) edge[->] (11);
\draw (7) edge[->] (12);
\draw (8) edge[->] (9);
\draw (8) edge[->] (10);
\draw (8) edge[->] (11);
\draw (8) edge[->] (12);
\draw (9) edge[->] (10);
\draw (9) edge[->] (11);
\draw (9) edge[->] (12);
\draw (11) edge[red,thick,->] (10);
\draw (10) edge[->] (12);
\draw (11) edge[->] (12);

\node[draw=none] () at (4.5,-.55) {\Large$\cg_2$};
\end{tikzpicture}
\caption{Two almost-complete DAGs that differ by a single super-covered edge and the missing edge is between two parents.}
\label{figure:MissingEdgeBetweenParents}
\end{figure}
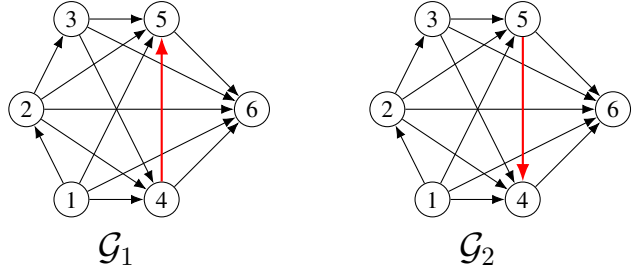

\begin{figure}
\begin{tikzpicture}[every node/.style={draw, circle, inner sep=2pt}, scale=1.1, every edge/.style={draw, >=Latex}]
\node (1) at (0,0) {1};
\node (2) at (1,-1) {2};
\node (3) at (1,1) {3};
\node (4) at (2.5,-1) {4};
\node (5) at (2.5,1) {5};

\draw (1) edge[->] (2);
\draw (1) edge[->] (3);
\draw (1) edge[->] (4);
\draw (1) edge[->] (5);
\draw (2) edge[red,thick,->] (3);
\draw (2) edge[->] (4);
\draw (2) edge[->] (5);
\draw (3) edge[->] (4);
\draw (3) edge[->] (5);

\node[draw=none] () at (1.6,-1.55) {\Large$\cg_1$};

\node (6) at (4,0) {1};
\node (7) at (5,-1) {2};
\node (8) at (5,1) {3};
\node (9) at (6.5,-1) {4};
\node (10) at (6.5,1) {5};

\draw (6) edge[->] (7);
\draw (6) edge[->] (8);
\draw (6) edge[->] (9);
\draw (6) edge[->] (10);
\draw (8) edge[red,thick,->] (7);
\draw (7) edge[->] (9);
\draw (7) edge[->] (10);
\draw (8) edge[->] (9);
\draw (8) edge[->] (10);

\node[draw=none] () at (5.6,-1.55) {\Large$\cg_2$};

\node (11) at (8,0) {1};
\node (12) at (9,-1) {2};
\node (13) at (9,1) {3};
\node (14) at (10.5,-1) {4};
\node (15) at (10.5,1) {5};

\draw (11) edge[->] (12);
\draw (11) edge[->] (13);
\draw (11) edge[->] (14);
\draw (11) edge[->] (15);
\draw (13) edge[red,thick,->,bend left=15] (12);
\draw (12) edge[red,thick,->,bend left=15] (13);
\draw (12) edge[->] (14);
\draw (12) edge[->] (15);
\draw (13) edge[->] (14);
\draw (13) edge[->] (15);

\node[draw=none] () at (9.6,-1.55) {\Large$\cg$};

\end{tikzpicture}
\caption{Two almost-complete DAGs that differ by a single super-covered edge and the missing edge is between two children, and $\cg=\cg_1\cup \cg_2$.}
\label{figure:MissingEdgeBetweenChildren}
\end{figure}
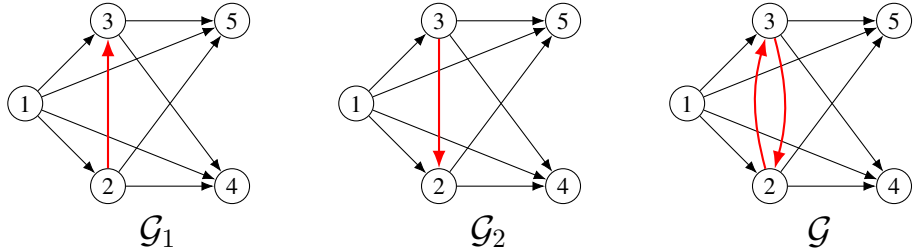

In this example, we illustrate the technique used in the proof of \cref{lemma: almost complete}. Consider the almost-complete DAGs $\cg_1$ and $\cg_2$ as shown in \cref{figure:MissingEdgeBetweenParents}, where $4\rightarrow 5$ is super-covered in $\cg_1$, which is flipped to obtain $\cg_2$. This setup corresponds to the first case in the proof of \cref{lemma: almost complete}, where the only missing edge is between two parents, $1$~and~$3$. Observe that the set $\{4,5,6\}$ must form a 3-clique in both $\cg_1$ and $\cg_2$, and the subgraphs $\cg_1[1,2,3]$ and $\cg_2[1,2,3]$ are identical. Thus, by \cref{lem:remove full sink clique} we can conclude that the models $\cm_{\cg_1}$ and $\cm_{\cg_2}$ are equal.

Now, consider the almost-complete DAGs $\cg_1$ and $\cg_2$ shown in \cref{figure:MissingEdgeBetweenChildren}, where the edge $2\rightarrow 3$ is super-covered in $\cg_1$ and is flipped to obtain $\cg_2$. This corresponds to the setup of the second case, where the only missing edge is between two children, $4$ and $5$. We first construct the non-simple graph $\cg = \cg_1\cup \cg_2$ which has the double edge between $2$ and $3$. The~idea here is to show that $\cg$ is not parameter identifiable and invoke \cref{lem:lyap_struct_id_from_param_id} to conclude that $\cg_1$ and $\cg_2$ are model equivalent. Hence we must show that the $15 \times 15$ matrix $A_{\cg}(\Sigma)$ does not have full rank. This is proved by constructing a non-zero element $D$ in its kernel as a linear combination of the columns of the matrix $H(\Sigma)$. In the notation of the proof, we have $P = \{1\}$,  $Q=\{1,2,3\}$, $T=\{(1,2),(1,3),(2,3)\}$. The crucial property of $D$ is that its entries indexed by non-edges of $\cg$ are zero. These entries are shown below together with the relevant rows of $H(\Sigma)$ and the non-zero entries in the coefficient vector $d$:
\begin{gather*}
\begin{pmatrix}
D_{2\to 1} \\
D_{3\to 1} \\
D_{4\to 1} \\
D_{5\to 1} \\
\vdots \\
D_{5\to 4}
\end{pmatrix} =
\begin{pmatrix}
\sigma_{11} & 0 & -\sigma_{13} \\
0 & \sigma_{11} & \sigma_{12} \\
0 & 0 & 0 \\
0 & 0 & 0 \\
\vdots & \vdots & \vdots \\
0 & 0 & 0
\end{pmatrix}
\begin{pmatrix}
-\sigma_{13}\\
\sigma_{12} \\
-\sigma_{11}
\end{pmatrix}.
\end{gather*}
All of these entries come out~as~zero, as required.
\end{example}

\subsection{Proofs of the Main Theorems}

We have now established that model equivalent DAGs have the same skeleton and model equivalent induced subgraphs and that flipping a super-covered edge in a DAG $\cg$ yields a new graph which is model equivalent to~$\cg$. To complete the proofs of \cref{thm:struct_id_induced_subgraph,theorem:super covered edge flip} we must now show that if $\cg_1$ and $\cg_2$ are DAGs such that $\cg_1[K]$ and $\cg_2[K]$ are model equivalent for all $K \subseteq V$ with $|K| = 4$, then $\cg_1$ and $\cg_2$ can be transformed into one another by a sequence of super-covered edge flips.

\begin{lemma}\label{lemma:super-covered-4-node-flips}
Let $\cg_1$ and $\cg_2$ be simple DAGs on 4 nodes. Then $\cg_1$ and $\cg_2$ are equivalent if and only if there exists a sequence of super-covered edge flips that transforms $\cg_1$ to $\cg_2$.
\end{lemma}

\begin{proof}
This result follows from direct computation. All 4-node graphs which are not pictured in \cref{fig:NonIdent} do not contain a super-covered edge and thus are structurally identifiable.
On~the other hand, in each of the seven fundamental types of non-identifiable Lyapunov DAG models pictured in \cref{fig:NonIdent}, the model equivalence class can be generated by starting with any member and performing all possible sequences of super-covered edge flips.
\end{proof}

\begin{lemma}\label{lemma:small-to-big-super-covered}
Let $\cg = (V, E)$ be a simple DAG with $|V| \ge 4$. Then $i\to j \in E$ is super-covered in $\cg$ if and only if $i \to j$ is super-covered in every induced subgraph of the form $\cg[i, j, k, l]$ for $k, l \in V \setminus \{i, j\}$.
\end{lemma}
\begin{proof}

It is immediate from the definition of a super-covered edge that if edge $i \to j$ is super-covered in $\cg$ then it is super-covered in every induced subgraph containing $i$ and $j$ so it only remains to show the other direction.

The conditions $\ch_\cg(i) = \ch_\cg(j) \cup \{j\}$ and $\pa_\cg(i) \cup \{i\} = \pa_\cg(j)$ are easy to verify by inspection of appropriate induced subgraphs (even on three nodes at a time). For each $k \in \pa_\cg(j)$ and $l \in \ch_\cg(i)$ we have $k \to l \in E(\cg[i,j,k,l])$ by assumption and hence $k \to l \in E$.
It remains to show that any $k \in V \setminus\{i, j\}$ which is not marginally independent of $\{i,j\}$ must be adjacent to one (and therefore both) of them. Suppose that $k \not \indep \{i,j\}$ in $\cg$ meaning that there exists a trek from $k$ to $i$ or $j$. Since $i$ and $j$ have the same parents, we may assume without loss of generality that there is a trek from $k$ to $i$ of the form $k = k_p \ot k_{p-1} \ot \cdots \ot k_0 = l_0 \to \cdots \to l_{q-2} \to l_{q-1} \to l_q = i$. Then it holds that $l_{q-2} \to i \in E$ since $i \to j$ is super-covered in $E(\cg[i, j, l_{q-1}, l_{q-2}])$. Hence, we may shorten the trek by leaving out $l_{q-1}$. Iterating this argument shows that $k$ must be a parent or a child of $i$ which completes the proof.
\end{proof}

\begin{theorem}\label{thm:4-to-infinity}
Let $\cg_1$ and $\cg_2$ be two DAGs on vertex set~$V$, with $|V| \ge 4$, such that $\cm_{\cg_1[K]} = \cm_{\cg_2[K]}$ for any subset $K\subseteq V$ of size four. Then $\cg_1$ and $\cg_2$ have the same skeleton and there exists a sequence of $\delta$ super-covered edge flips which transforms $\cg_1$ into $\cg_2$, where $\delta = |E(\cg_1) \setminus E(\cg_2)|$. In particular $\cm_{\cg_1} = \cm_{\cg_2}$.
\end{theorem}

\begin{proof}
As $\cg_1[K] \sim \cg_2[K]$ for any subset $K \subseteq V$ of size four, by \cref{lemma:skel_struct_id} it follows that $\cg_1[K]$ and $\cg_2[K]$ have the same skeleton and thus $\cg_1$ and $\cg_2$ have the same skeleton as~well.
We employ the \textbf{Find-Edge} algorithm from \cite{chickering1995transformational} and show that under our stronger assumptions on the graphs $\cg_1$ and $\cg_2$ the edge it outputs is not only covered but super-covered. Order the edges of $\cg_1$ by the rule: $(i \to j) \prec (k \to l)$ if and only if $j < l$, or $j = l$ and $i > k$ in the topological ordering on $\cg_1$. Let $\Delta(\cg_1, \cg_2) = E(\cg_1) \setminus E(\cg_2)$ denote the set of edges which need to be flipped.

We claim that the $\prec$-minimal edge $i \to j$ in $\Delta(\cg_1, \cg_2)$ is super-covered in $\cg_1$. Suppose that it is not. Then \cref{lemma:small-to-big-super-covered} yields $k, l \in V$ such that $i \to j$ is not super-covered in $\cg_1[i,j,k,l]$. On the other hand, $\cg_1[i,j,k,l] \sim \cg_2[i,j,k,l]$ and since these graphs are model equivalent but not equal, there must exist a super-covered edge in $\cg_1[i,j,k,l]$ by \cref{lemma:super-covered-4-node-flips}. Flipping it produces a third distinct member of the model equivalence class of this~DAG. Thus, $\cg_1[i,j,k,l]$ must be a 4-node DAG in which $i \to j$ is not super-covered but some other edge is. Of all the non-identifiable 4-node graphs in \cref{fig:NonIdent} these criteria narrow the possibilities down to the types (\subref{fig:NonIdent4}) and (\subref{fig:NonIdent5}):
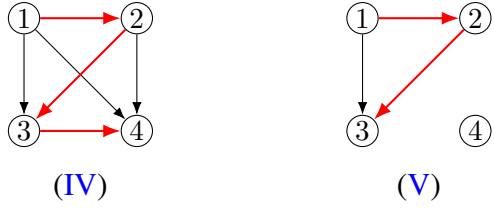
\begin{center}
\begin{tikzpicture}[scale=1.5, every node/.style={draw, circle, inner sep=1pt}, every edge/.style={draw, >=Latex}]
  \node (11) at (0,0) {$1$}; \node (21) at (1,0) {$2$};
  \node (31) at (0,-1) {$3$}; \node (41) at (1,-1) {$4$};
  \draw (11) edge[->,red,thick] (21); \draw (11) edge[->] (31); \draw (11) edge[->] (41);
  \draw (21) edge[->,red,thick] (31); \draw (21) edge[->] (41); \draw (31) edge[->,red,thick] (41);
  \node[draw=none,rectangle] (t1) at (0.5,-1.5) {\large(\subref{fig:NonIdent4})};

  \node (12) at (3,0) {$1$}; \node (22) at (4,0) {$2$};
  \node (32) at (3,-1) {$3$}; \node (42) at (4,-1) {$4$};
  \draw (12) edge[->,red,thick] (22); \draw (12) edge[->] (32); \draw (22) edge[->,red,thick] (32);
  \node[draw=none,rectangle] (t2) at (3.5,-1.5) {\large(\subref{fig:NonIdent5})};
\end{tikzpicture}%
\end{center}
The super-covered edges in both of these graphs are highlighted in red. Thus we consider all assignments of $i,j,k,l$ to $1,2,3,4$ such that $i \to j$ exists but is not super-covered. In case (\subref{fig:NonIdent5}) there is only one possibility which puts $i \to j$ as $1 \to 3$. The edges $1\to 2$ and $2\to 3$ are both $\prec$-smaller than $1\to 3$, so by the choice of $i \to j$ as the minimal element of $\Delta(\cg_1, \cg_2)$ we know that these two edges have the same orientation in $\cg_2$. But then $3 \to 1$ in $\cg_2$ creates a cycle which is absurd.
For type (\subref{fig:NonIdent4}) we have three possibilities which all lead to contradictions in the same way:
\begin{itemize}
\item If $i \to j$ is $1 \to 3$, we have the same situation as for type (\subref{fig:NonIdent5}).
\item If $i \to j$ is $1 \to 4$, then $1\to 2$ and $2 \to 4$ are $\prec$-smaller and thus must exist in $\cg_2$. This shows that $1 \to 2 \to 4 \to 1$ is a cycle in $\cg_2$.
\item If $i \to j$ is $2 \to 4$, then $2 \to 3$ and $3 \to 4$ exist in $\cg_2$ because they are $\prec$-smaller than $2 \to 4$. Again, this gives a cycle $2 \to 3 \to 4 \to 2$ in $\cg_2$.
\end{itemize}
Ultimately, the assumption that $i \to j$ is not super-covered leads to a contradiction. So $i \to j$ is super-covered and we may flip it, producing a DAG $\cg_1'$ which is model equivalent to $\cg_1$. It satisfies the assumptions of our claim and has $|\Delta(\cg_1', \cg_2)| < |\Delta(\cg_1, \cg_2)|$. Thus, the claim follows by induction.
\end{proof}

\begin{proof}[Proof of \cref{thm:struct_id_induced_subgraph}]
If $\cg_1[K] \sim \cg_2[K]$ for all $K \subseteq V$ such that $|K| = 4$, then \cref{thm:4-to-infinity} yields $\cg_1 \sim \cg_2$. On the other hand, if there exists any set $K \subseteq V$ such that $\cg_1[K] \not \sim \cg_2[K]$ then $\cg_1 \not \sim \cg_2$ by \cref{prop:induced_subgraph_distinguish}.
\end{proof}

\begin{proof}[Proof of \cref{theorem:super covered edge flip}]
If there exists a sequence of super-covered edge flips which transforms $\cg_1$ into $\cg_2$ then \cref{prop: edge flip} implies $\cg_1 \sim \cg_2$. Now suppose that $\cg_1 \sim \cg_2$. By \cref{thm:struct_id_induced_subgraph} we have that $\cg_1[K] \sim \cg_2[K]$ for all $K \subseteq V$ with $|K|= 4$. The desired result then follows immediately from \cref{thm:4-to-infinity}.
\end{proof}

\begin{proof}[Proof of \cref{thm:struct-id-forbidden-minors}]
A DAG~$\cg$ is structurally identifiable within the class of DAGs if and only if it does not contain a super-covered edge. \Cref{lemma:small-to-big-super-covered} establishes that this property holds for $\cg$ if and only if it holds for all its induced subgraphs~$\cg[K]$ with $|K|=4$. \Cref{fig:NonIdent} lists all isomorphism types of DAGs on four nodes containing a super-covered~edge.
\end{proof}

\section{Discussion}\label{sec:outlook}

In this section we summarize our results in a broader context, discuss their limitations and highlight directions for future work.

\subsection{Lyapunov as a causal refinement of Bayes}

A Lyapunov DAG model $\cm_\cg$ is defined by a stochastic differential equation analogous to the linear structural equation of a Gaussian Bayesian network. We~gave a complete solution to the model equivalence problem for Lyapunov DAG models. This solution has two components:
\begin{enumerate}
\item A graph-theoretic characterization via induced subgraphs (\cref{thm:struct_id_induced_subgraph}) which parallels the v-structure criterion of Verma and Pearl for Bayesian networks~\cite{VermaPearl};
\item a transformational characterization via super-covered edge flips (\cref{theorem:super covered edge flip}) in analogy to the covered edge flips of Chickering~\cite{chickering1995transformational}.
\end{enumerate}

\begin{table}
\begin{tabular}{c|c||c|c||c|c}
\multicolumn{2}{c}{} & \multicolumn{2}{c}{Lyapunov DAG models} & \multicolumn{2}{c}{Bayesian network models} \\
$n$ & DAGs          & Distinct models & Identifiable  & Distinct models & Identifiable \\ \hline
$3$ & $25$          & $17$            & $13$          & $11$            & $4$ \\ \hline
$4$ & $543$         & $461$           & $423$         & $185$           & $59$ \\ \hline
$5$ & $29\,281$     & $27\,697$       & $26\,761$     & $8\,782$        & $2\,616$ \\ \hline
$6$ & $3\,781\,503$ & $3\,715\,745$   & $3\,665\,673$  & $1\,067\,825$   & $306\,117$
\end{tabular}
\caption{The numbers of model equivalence classes and structurally identifiable models on graphs with few vertices.}
\label{tab:Ident}
\end{table}

As an unexpected result of this parallel development of model equivalence criteria we are able to directly compare Lyapunov DAG and Bayesian network models: any two DAGs which have the same Lyapunov model also have the same Bayesian network model (\cref{cor:lyap-refines-bayes}). In other words, Lyapunov models provide a \emph{causal refinement} of Bayesian networks. They have smaller equivalence classes and are more often structurally identifiable. A quantitative comparison between the two model classes is given in \Cref{tab:Ident}.
\begin{problem}
Obtain general quantitative bounds on the share of structurally identifiable graphs and the average size of a Lyapunov model equivalence class.
\end{problem}
Asymptotic results of this nature for Bayesian networks have been obtained in \cite{AsymptoticMEC}. Our~graph-theoretic characterizations of model equivalence should prove useful in this investigation for Lyapunov models.

First results in causal discovery in the setting of diffusion processes were achieved in \cite{lorch2024causal}. Our~point of view creates additional opportunities to import well-established and successful ideas from Bayesian networks. We highlight the following task:

\begin{problem}
Adapt the Greedy Equivalence Search (GES~\cite{GES}) algorithm to Lyapunov DAG models.
\end{problem}

\subsection{Model-defining equations and conditional independence}

Our theorems build on a description of Lyapunov models via missing-edge relations which are derived from the Lyapunov equation by Cramer's rule. As such they are polynomials of high degree and difficult to use in computational practice. Moreover, we cannot offer a statistical interpretation of these~relations.

\begin{problem}
Find more efficient (lower degree) relations with a clear statistical interpretation which define Lyapunov models in the sense of \cref{prop:MissingEdge}.
\end{problem}

By \cref{rem:MissingEdge} missing-edge relations in Bayesian networks do have a straightforward statistical interpretation: they correspond to conditional independence constraints. This does not carry over to the Lyapunov realm (cf.~\cite{drton2024conditional}) which is what accounts for some of the technical difficulties in our proofs compared to the Bayesian network literature. The following observation gives an idea of how far away Lyapunov models are from being specified by their conditional independence structure:

\begin{proposition} \label{prop:ci-eq-model}
Let $\cg$ be a DAG. The model $\cm_\cg$ is defined by its conditional independence structure if and only if for each missing edge $i \to j \notin E(\cg)$ there is no trek between between $i$ and $j$ in~$\cg$.
\end{proposition}

This follows without much difficulty from the main theorem in \cite{drton2024conditional} and \cref{prop:Geom} using dimension considerations. Note that a DAG $\cg$ satisfying the condition in \cref{prop:ci-eq-model} must have either $i \to j \in E$ or $j \to i \in E$ for every pair of vertices $i, j$ which are connected by a trek. In particular it must be a transitive DAG. This condition fails already for the 3-path.

We mention in this context another observation first made in \cite[Section~4.5]{boege2024realbirational}. Consider the 3-path $1 \to 2 \to 3$ from \cref{eg:123} whose Lyapunov model satisfies no conditional independence. However, when the model is restricted to have normed variances, i.e., $\sigma_{11} = \sigma_{22} = \sigma_{33} = 1$, then the missing-edge relation \eqref{eq:123} simplifies to
\[
  (\sigma_{13} - \sigma_{12} \sigma_{23}) \cdot (1 - \sigma_{12}\sigma_{13}\sigma_{23}) = 0,
\]
which, for correlation matrices, is equivalent to the conditional independence $[1 \indep 3 \mid 2]$. This~is precisely the conditional independence defining the Bayesian network of the 3-path. This curious fact is not true for all graphs as a computation with the tree $2 \ot 1 \to 3$ shows.

\begin{problem}
Characterize the conditional independence structures of Lyapunov correlation models. What is their relation to d-separation?
\end{problem}

\subsection{Cyclic models and correlated noise}

Our structural identifiability results hold within the class of Lyapunov DAG models with uncorrelated noise. A natural question is how far beyond these assumptions our techniques can be pushed --- in particular because one of the main promises of GCLMs is the ability to model causal feedback loops (i.e., cycles) without the loss of favorable properties such as parameter identifiability.

\begin{problem}
Extend the model equivalence characterizations to simple cyclic graphs and to the correlated noise setting.
\end{problem}

First note that all results in \cref{sec:Tools,sec:Subgraphs} also apply to cyclic graphs, as long as they are simple. Together they prove that model equivalent graphs have the same skeleton and model equivalent induced subgraphs. The acyclicity assumption enters in the proof of the converse in \cref{sec:EdgeFlips} via super-covered edge flips; it is essential in the proof of \cref{thm:4-to-infinity}. For dimension reasons, the 3-cycle and any complete DAG on three vertices have the same model, namely the whole space~$\PD_3$. Yet, every super-covered edge flip of a complete DAG produces another complete DAG. This shows that the concept of a super-covered edge is not flexible enough to connect all model equivalent graphs when cycles are allowed.
For Bayesian networks, the graphical model equivalence characterization has been extended to include cyclic graphs but this required significant work and new concepts; cf.~\cite{RichardsonCyclic}.

The assumption of uncorrelated noise, i.e., that the diffusion matrix $D$ in \eqref{eqn:ornstein-uhlenbeck} and hence also $C = D D^T$ is diagonal, can be loosened somewhat. This assumption is needed only to invoke (in several places) the result of \cite[Corollary 2.3]{varando2020graphical} (also \cite[Proposition~8.3]{dettling2022identifiability}) that the non-existence of a trek between $i$~and~$j$ implies the marginal independence~$i \indep j$. A recent generalization due to Hansen~\cite[Proposition~2.3]{hansen2025trek} establishes this implication for all matrices $C$ such that if there is no trek between $i$ and $j$ in $\cg$ then $c_{kl} = c_{lk} = 0$ for all ancestors $k$ of $i$ and all ancestors $l$ of $j$ in~$\cg$.
If the same $C$ satisfies this condition for two DAGs $\cg_1, \cg_2$ then our graphical criteria also characterize the model equivalence $\cm_{\cg_1,C} = \cm_{\cg_2,C}$. If the class of graphs under consideration is sufficiently broad --- such as all DAGs --- then the conjunction of all the conditions narrows the choices for $C$ down to diagonal matrices, which is the uncorrelated noise version presented here. In the class of polytrees, however, any two vertices are always connected by a trek and the conditions on $C$ are vacuous. Thus, at the expense of reducing the graph class from DAGs to polytrees, our techniques apply to general diffusion matrices.

\begin{funding}
T.B.\ was funded by the European Union's Horizon 2020 research and innovation programme under the Marie Skłodowska-Curie grant agreement No.~101110545. B.H. was supported by the Alexander von Humboldt Foundation. P.M. received funding from the European Research Council (ERC) under the European Union’s Horizon 2020 research and innovation programme (grant agreement No. 883818).
\end{funding}

\bibliographystyle{imsart-number} %
\bibliography{refs}       %

\begin{thebibliography}{18}

\bibitem{andersson1997characterization}
\begin{barticle}[author]
\bauthor{\bsnm{Andersson},~\bfnm{Steen~A.}\binits{S.~A.}},
  \bauthor{\bsnm{Madigan},~\bfnm{David}\binits{D.}} \AND
  \bauthor{\bsnm{Perlman},~\bfnm{Michael~D.}\binits{M.~D.}}
(\byear{1997}).
\btitle{A characterization of {Markov} equivalence classes for acyclic
  digraphs}.
\bjournal{Ann. Stat.}
\bvolume{25}
\bpages{505--541}.
\bdoi{10.1214/aos/1031833662}
\end{barticle}
\endbibitem

\bibitem{drton2024conditional}
\begin{barticle}[author]
\bauthor{\bsnm{Boege},~\bfnm{Tobias}\binits{T.}},
  \bauthor{\bsnm{Drton},~\bfnm{Mathias}\binits{M.}},
  \bauthor{\bsnm{Hollering},~\bfnm{Benjamin}\binits{B.}},
  \bauthor{\bsnm{Lumpp},~\bfnm{Sarah}\binits{S.}},
  \bauthor{\bsnm{Misra},~\bfnm{Pratik}\binits{P.}} \AND
  \bauthor{\bsnm{Schkoda},~\bfnm{Daniela}\binits{D.}}
(\byear{2025}).
\btitle{{Conditional independence in stationary diffusions}}.
\bjournal{Stochastic Processes and their Applications}.
\bdoi{10.1016/j.spa.2025.104604}
\end{barticle}
\endbibitem

\bibitem{boege2024realbirational}
\begin{bmisc}[author]
\bauthor{\bsnm{Boege},~\bfnm{Tobias}\binits{T.}} \AND
  \bauthor{\bsnm{Solus},~\bfnm{Liam}\binits{L.}}
(\byear{2024}).
\btitle{Real birational implicitization for statistical models}.
\bhowpublished{Preprint, {arXiv}:2410.23102 [math.{ST}]}.
\end{bmisc}
\endbibitem

\bibitem{chickering1995transformational}
\begin{binproceedings}[author]
\bauthor{\bsnm{Chickering},~\bfnm{David~Maxwell}\binits{D.~M.}}
(\byear{1995}).
\btitle{A transformational characterization of equivalent Bayesian network
  structures}.
In \bbooktitle{Proceedings of the Eleventh conference on Uncertainty in
  artificial intelligence}
\bpages{87--98}.
\end{binproceedings}
\endbibitem

\bibitem{GES}
\begin{barticle}[author]
\bauthor{\bsnm{Chickering},~\bfnm{David~Maxwell}\binits{D.~M.}}
(\byear{2003}).
\btitle{Optimal structure identification with greedy search}.
\bjournal{J. Mach. Learn. Res.}
\bvolume{3}
\bpages{507--554}.
\bdoi{10.1162/153244303321897717}
\end{barticle}
\endbibitem

\bibitem{CLO}
\begin{bbook}[author]
\bauthor{\bsnm{Cox},~\bfnm{David~A.}\binits{D.~A.}},
  \bauthor{\bsnm{Little},~\bfnm{John}\binits{J.}} \AND
  \bauthor{\bsnm{O'Shea},~\bfnm{Donal}\binits{D.}}
(\byear{2015}).
\btitle{Ideals, varieties, and algorithms},
\bedition{4th} ed.
\bseries{Undergraduate Texts in Mathematics}.
\bpublisher{Springer}.
\bdoi{10.1007/978-3-319-16721-3}
\end{bbook}
\endbibitem

\bibitem{dettling2022lasso}
\begin{binproceedings}[author]
\bauthor{\bsnm{Dettling},~\bfnm{Philipp}\binits{P.}},
  \bauthor{\bsnm{Drton},~\bfnm{Mathias}\binits{M.}} \AND
  \bauthor{\bsnm{Kolar},~\bfnm{Mladen}\binits{M.}}
(\byear{2024}).
\btitle{On the {Lasso} for Graphical Continuous {Lyapunov} Models}.
In \bbooktitle{Proceedings of the Third Conference on Causal Learning and
  Reasoning}
(\beditor{\bfnm{Francesco}\binits{F.}~\bsnm{Locatello}} \AND
  \beditor{\bfnm{Vanessa}\binits{V.}~\bsnm{Didelez}}, eds.).
\bseries{Proceedings of Machine Learning Research}
\bvolume{236}
\bpages{514--550}.
\bpublisher{PMLR}.
\end{binproceedings}
\endbibitem

\bibitem{dettling2022identifiability}
\begin{barticle}[author]
\bauthor{\bsnm{Dettling},~\bfnm{Philipp}\binits{P.}},
  \bauthor{\bsnm{Homs},~\bfnm{Roser}\binits{R.}},
  \bauthor{\bsnm{Am{\'e}ndola},~\bfnm{Carlos}\binits{C.}},
  \bauthor{\bsnm{Drton},~\bfnm{Mathias}\binits{M.}} \AND
  \bauthor{\bsnm{Hansen},~\bfnm{Niels~Richard}\binits{N.~R.}}
(\byear{2023}).
\btitle{Identifiability in continuous {Lyapunov} models}.
\bjournal{SIAM J. Matrix Anal. Appl.}
\bvolume{44}
\bpages{1799--1821}.
\bdoi{10.1137/22M1520311}
\end{barticle}
\endbibitem

\bibitem{drton2023identifiabilityhomoscedasticlinearstructural}
\begin{barticle}[author]
\bauthor{\bsnm{Drton},~\bfnm{Mathias}\binits{M.}},
  \bauthor{\bsnm{Hollering},~\bfnm{Benjamin}\binits{B.}} \AND
  \bauthor{\bsnm{Wu},~\bfnm{Jun}\binits{J.}}
(\byear{2025}).
\btitle{Identifiability of homoscedastic linear structural equation models
  using algebraic matroids}.
\bjournal{Adv. Appl. Math.}
\bvolume{163}
\bpages{26}.
\bnote{Id/No 102794}.
\bdoi{10.1016/j.aam.2024.102794}
\end{barticle}
\endbibitem

\bibitem{katie2019}
\begin{bmisc}[author]
\bauthor{\bsnm{Fitch},~\bfnm{Katherine~E.}\binits{K.~E.}}
(\byear{2019}).
\btitle{Learning Directed Graphical Models from {G}aussian Data}.
\bhowpublished{Preprint, {arXiv}:1906.08050 [cs.{LG}]}.
\end{bmisc}
\endbibitem

\bibitem{hansen2025trek}
\begin{barticle}[author]
\bauthor{\bsnm{Hansen},~\bfnm{Niels~Richard}\binits{N.~R.}}
(\byear{2025}).
\btitle{A trek rule for the Lyapunov equation}.
\bjournal{Algebr. Stat.}
\bvolume{16}
\bpages{95--112}.
\bdoi{10.2140/astat.2025.16.95}
\end{barticle}
\endbibitem

\bibitem{lauritzen1996graphical}
\begin{bbook}[author]
\bauthor{\bsnm{Lauritzen},~\bfnm{Steffen~L}\binits{S.~L.}}
(\byear{1996}).
\btitle{Graphical models}
\bvolume{17}.
\bpublisher{Clarendon Press}.
\end{bbook}
\endbibitem

\bibitem{lorch2024causal}
\begin{binproceedings}[author]
\bauthor{\bsnm{Lorch},~\bfnm{Lars}\binits{L.}},
  \bauthor{\bsnm{Krause},~\bfnm{Andreas}\binits{A.}} \AND
  \bauthor{\bsnm{Sch{\"o}lkopf},~\bfnm{Bernhard}\binits{B.}}
(\byear{2024}).
\btitle{Causal Modeling with Stationary Diffusions}.
In \bbooktitle{Proc. International Conference on Artificial Intelligence and
  Statistics (AISTATS)}
\bpages{1927--1935}.
\bpublisher{PMLR}.
\end{binproceedings}
\endbibitem

\bibitem{RichardsonCyclic}
\begin{barticle}[author]
\bauthor{\bsnm{Richardson},~\bfnm{Thomas}\binits{T.}}
(\byear{1997}).
\btitle{A characterization of {Markov} equivalence for directed cyclic graphs}.
\bjournal{Int. J. Approx. Reasoning}
\bvolume{17}
\bpages{107--162}.
\bdoi{10.1016/S0888-613X(97)00020-0}
\end{barticle}
\endbibitem

\bibitem{AsymptoticMEC}
\begin{bmisc}[author]
\bauthor{\bsnm{Schmid},~\bfnm{Dominik}\binits{D.}} \AND
  \bauthor{\bsnm{Sly},~\bfnm{Allan}\binits{A.}}
(\byear{2022}).
\btitle{On the number and size of {Markov} equivalence classes of random
  directed acyclic graphs}.
\bhowpublished{Preprint, {arXiv}:2209.04395 [math.{PR}]}.
\end{bmisc}
\endbibitem

\bibitem{sullivant2018algebraic}
\begin{bbook}[author]
\bauthor{\bsnm{Sullivant},~\bfnm{Seth}\binits{S.}}
(\byear{2018}).
\btitle{Algebraic statistics}.
\bseries{Graduate Studies in Mathematics}
\bvolume{194}.
\bpublisher{American Mathematical Society, Providence, RI}.
\end{bbook}
\endbibitem

\bibitem{varando2020graphical}
\begin{binproceedings}[author]
\bauthor{\bsnm{Varando},~\bfnm{Gherardo}\binits{G.}} \AND
  \bauthor{\bsnm{Hansen},~\bfnm{Niels~Richard}\binits{N.~R.}}
(\byear{2020}).
\btitle{Graphical continuous Lyapunov models}.
In \bbooktitle{Proceedings of the 36th Conference on Uncertainty in Artificial
  Intelligence (UAI)}
(\beditor{\bfnm{Jonas}\binits{J.}~\bsnm{Peters}} \AND
  \beditor{\bfnm{David}\binits{D.}~\bsnm{Sontag}}, eds.).
\bseries{Proceedings of Machine Learning Research}
\bvolume{124}
\bpages{989--998}.
\bpublisher{PMLR}.
\end{binproceedings}
\endbibitem

\bibitem{VermaPearl}
\begin{binproceedings}[author]
\bauthor{\bsnm{Verma},~\bfnm{Thomas}\binits{T.}} \AND
  \bauthor{\bsnm{Pearl},~\bfnm{Judea}\binits{J.}}
(\byear{1990}).
\btitle{Equivalence and synthesis of causal models}.
In \bbooktitle{Proceedings of the Sixth Annual Conference on Uncertainty in
  Artificial Intelligence}.
\bseries{UAI '90}
\bpages{255--270}.
\bpublisher{Elsevier}.
\end{binproceedings}
\endbibitem

\end{thebibliography}

\end{document}